\PassOptionsToPackage{dvipsnames}{xcolor}
\documentclass[11pt,a4paper]{article}
\usepackage[left=0.5in, right=0.5in, top=0.5in, bottom=0.5in]{geometry}

\usepackage{graphicx}%
\usepackage{multirow}%
\usepackage{amsmath,amssymb,amsfonts}%
\usepackage{amsthm}
\usepackage{mathrsfs}%
\usepackage[title]{appendix}%
\usepackage{xcolor}%
\usepackage{textcomp}%
\usepackage{manyfoot}%
\usepackage{booktabs}%
\usepackage{algorithm}%
\usepackage{algorithmicx}%
\usepackage{algpseudocode}%
\usepackage{listings}%
\usepackage{hyphenat}
\usepackage{bbm}

\usepackage{tikz}

\usepackage{pgfplots}
\pgfplotsset{compat=1.18}
\usepgfplotslibrary{statistics}
\usepackage{float}
\usepackage{placeins} % provides \FloatBarrier
\usepackage{subcaption}

\usepackage[shortlabels]{enumitem}
\usepackage{hyperref}
\usepackage{authblk}

\usepackage{definitions}

\usepackage{array}
\newcounter{rowcntr}[table]
\renewcommand{\therowcntr}{\thetable.\arabic{rowcntr}}
\newcolumntype{N}{>{\refstepcounter{rowcntr}\therowcntr}c}
\AtBeginEnvironment{tabular}{\setcounter{rowcntr}{0}}

\newtheorem{theorem}{Theorem}[section]
\newtheorem{lemma}[theorem]{Lemma}
\newtheorem{corollary}[theorem]{Corollary}
\newtheorem{proposition}[theorem]{Proposition}
\newtheorem{definition}[theorem]{Definition}
\newtheorem{remark}[theorem]{Remark}

\title{Model order reduction of an ultraweak and optimally stable variational formulation for
parametrized reactive transport problems}

\author[1]{Christian Engwer}
\author[1]{Mario Ohlberger}
\author[1]{\underline{Lukas Renelt}}
\affil[1]{University of M\"unster, Institute for Analysis and Numerics, Einsteinstr. 62, M\"unster, 48149, Germany,
\{christian.engwer, mario.ohlberger, lukas.renelt\}@uni-muenster.de}

\begin{document}

\makeatletter
\pgfplotsset{
  boxplot/hide outliers/.code={
    \def\pgfplotsplothandlerboxplot@outlier{}%
  }
}
\makeatother

\maketitle

\begin{abstract}
  This contribution introduces a model order reduction approach for an advection-reaction
  problem with a parametrized reaction function.
  The underlying discretization uses an ultraweak formulation with an $L^2$-like trial space and an
  `optimal' test space as introduced by Demkowicz et al. This ensures the stability of the
  discretization and in addition allows for a symmetric reformulation of the problem in terms of a dual
  solution which can also be interpreted as the normal equations of an adjoint least-squares problem.
  Classic model order reduction techniques can then be applied to
  the space of dual solutions which also immediately gives a reduced primal space.
  We show that the necessary computations do not require the reconstruction of any
  primal solutions and can instead be performed entirely on the space of dual solutions. We prove
  exponential convergence of the Kolmogorov $N$-width and show that a greedy algorithm produces
  quasi-optimal approximation spaces for both the primal and the dual solution space.
  Numerical experiments based on the benchmark problem of a catalytic filter confirm the applicability
  of the proposed method.
\end{abstract}

\section{Introduction}\label{sec:introduction}
Problems involving transport phenomena occur in many applied sciences \ie physics, medicine, geosciences
or chemistry. While in many cases there exist well-established numerical schemes to compute approximate
solutions to these problems, the computational effort required can become quite large in realistic
scenarios. This is a particular problem if the underlying equation additionally depends on a set of
parameters and solutions are required for many different parameter combinations. Some classic examples
of such a requirement include real-time-feasible simulations, PDE-constrained optimization or optimal
control problems. Projection-based model order reduction (MOR) has become an established technique and
was already successfully applied to many problem classes. In essence, one tries to construct a
low-dimensional \textit{reduced space} approximating solutions for all possible parameter values.
For a given parameter the computation of a solution in this reduced space is then very efficient
with a provable upper bound on the error compared to a high-dimensional solution, see \eg\cite{BennerOhlbergerCohenWillcox} for a general introduction.

In this paper we will present a model order reduction approach for the
parametrized stationary advection-reaction equation
with a given (parameter independent) transport field $\vec{b}$
\begin{equation}\label{eq:introduction:strongTransportEquation}
  \begin{cases}
    \nabla\cdot(\vec{b}u_\mu) + c_\mu u_\mu &= f_\mu \qquad\text{in}\;\Omega, \\
    \hfill u_\mu &= g_\mu \qquad \text{on}\;\Gamma_{\text{inflow}},
  \end{cases}
\end{equation}
although, the presented concepts can be applied to many other equation types, \eg to the class
of Friedrichs'-systems~\cite{friedrichs1958,romor2023friedrichs}. It is well-known that
classic projection-based reduction methods can fail to perform well if the transport field $\vec{b}$
was dependent on the parametrization. This is due to a slow decrease of the Kolmogorov $N$-width,
\ie the approximation error by the best possible $N$-dimensional linear subspace.
Slowly decreasing lower bounds on the $N$-width are rigorously known for
linear transport~\cite{OhlbergerRave,arbes2023kolmogorov}
and the parameterized wave-equation~\cite{greif2019kolmogorov}
and have been numerically observed for many nonlinear equations.
For these problems, nonlinear reduction approaches for example using convolutional
autoencoders~\cite{hesthaven2018nonintrusive,LeeCarlberg2020,kim2020nonlinear,romor2023nonlinear},
shifted proper orthogonal decomposition~\cite{reiss2018spod, burela2023spod},
implicit feature-tracking~\cite{mirhoseini2023model},
registration-based approaches~\cite{ferrero2022registration}
and many more have been developed in the last years.

In this contribution we focus on applications where the transport field is
given by a physical law and thus not necessarily part of the parametrization
($\vec{b}\neq\vec{b}_\mu$). One such problem that we
want to consider is the reactive transport of a substance in a catalytic filter. Here, the velocity
field can be obtained from the theory of porous media \eg as the solution to the Darcy-equation
(see \eg\cite{philip1970flow}).
It is then of interest to compute solutions with parametrized reaction coefficients, inflow
data, etc. In this case, we would expect a better approximability of the set of solutions although, at
least to our knowledge, theory is still scarce.
We will show that for linear advection with parametrized reaction the Kolmogorov $N$-width decays very
fast, in fact, even exponentially.
Such results are known in general for parameter-separable, linear and coercive
problems~\cite{OhlbergerRave} or inf-sup-stable problems~\cite{urban2023reduced}.
\par
In comparison to elliptic or parabolic problems, the discretization of hyperbolic PDEs is less
straightforward. Most of the existing methods can be written in the variational form
\begin{equation}\label{eq:introduction:generalProblem}
  \text{Find}\; u^\delta\in \xcal^\delta:\quad a^\delta(u^\delta,v^\delta) = f(v^\delta) \qquad \forall v^\delta\in \ycal^\delta.
\end{equation}
where the infinite-dimensional trial/test spaces $\xcal,\ycal$ have been discretized by
finite-dimensional spaces $\xcal^\delta,\ycal^\delta$.
These methods can then be divided further into Galerkin-methods ($\xcal^\delta = \ycal^\delta$)
and Petrov-Galerkin-methods ($\xcal^\delta \neq \ycal^\delta$).
An example of the former utilizing a continuous trial/test space is the
streamline-diffusion method (SDFEM) that introduces a stabilizing diffusive term in the direction of
transport~\cite{hughes1979sdfem} or Galerkin least-squares methods~\cite{hughes1989gls}.
Another broad class of methods employs discontinuous function spaces
which generally requires the use of a numerical flux on the cell interfaces. For linear advection, the
Riemann-problem for the flux can be solved explicitly resulting in the so-called upwind-scheme where the
flux on the interface is entirely given by the information from the upstream cell.

In the Petrov-Galerkin setting the discrete inf-sup-condition
\begin{equation}\label{intro:infSupCondition}
  \Inf_{u^\delta\in \xcal^\delta}\sup_{v^\delta\in \ycal^\delta}\frac{|a^\delta(u^\delta,v^\delta)|}{\norm{u^\delta}_{\xcal}\norm{v^\delta}_{\ycal}} =: \gamma^\delta > 0
\end{equation}
gives a criterion when a combination of trial and test spaces is stable~\cite{babuska1971error}.
In most cases one will for a
given trial space $\xcal^\delta$ augment the test space to ensure that~\eqref{intro:infSupCondition} is
satisfied. As an example, the additional diffusive contribution of the SDFEM can also be introduced by
altering the test space which is then known as the Streamline Upwind Petrov-Galerkin (SUPG)
method~\cite{brooks1982supg}. In the best case one would like to include for every
basis function $u^\delta\in \xcal^\delta$ the corresponding supremizer $s(u^\delta)\in \ycal$ solving
\begin{equation*}
  (s(u^\delta), v^\delta)_{\ycal} \;=\; a^\delta(u^\delta,v^\delta)\quad\forall v^\delta\in \ycal^\delta,
  \quad\text{i.e.}\quad s(u^\delta) = R_{\ycal}^{-1}(a^\delta(u^\delta,\cdot))
\end{equation*}
where $R_{\ycal}: \ycal\rightarrow \ycal'$ denotes the Riesz-map.
This is equivalent to $s(u^\delta)$ realizing the supremum in~\eqref{intro:infSupCondition}.
A discretization with $\xcal^\delta\subseteq \xcal$ and $s(\xcal^\delta) =: \ycal^\delta\subseteq \ycal$ is then \textit{optimally stable}
\ie yields the optimal inf-sup-constant $\gamma^\delta=1$. Explicitly computing
the supremizers is, however, often computationally as expensive as solving the full problem and thus
hardly practicable. An exception is given by the Discontinuous Petrov Galerkin (DPG)
method~\cite{demkowiczDPG1,demkowiczDPG2,broersen2018transportDPG}.
By choosing a localizable norm on $\ycal$, the inversion of the Riesz-map amounts to local
block-inversions such that the computation of supremizers is actually feasible.
Another approach is presented in~\cite{DahmenPleskenWelper} where the optimal test space $s(\xcal^\delta)$
is approximated by so-called $\delta$-proximal
spaces yielding quasi-optimal stability. This in particular allows for a continuous discretization of
the trial space. \footnote{We will subsequently refer to this approach as \textit{Continuous}
Petrov-Galerkin method or CPG for short.}

Applying model order reduction techniques to Petrov-Galerkin schemes is a challenging task, since
one needs to generate two different reduced spaces. In particular, the stability of the high-dimensional
model is not automatically inherited by the reduced problem and supremizer enrichment requires
parameter-dependent, high-dimensional computations.
Therefore, reduced test and trial spaces need to be generated
simultaneously, continuously ensuring their stability (which has for example been done in the
double-greedy method~\cite{DahmenHuangSchwab,DahmenPleskenWelper}).

Recently, an adjoint method to the DPG, called DPG*, has been proposed in~\cite{demkowiczDPGstar}.
Likewise, there also exists an adjoint method to the CPG which has in the optimal-stability context
first been proposed in~\cite{BrunkenSmetanaUrban} for linear transport and was since then also applied
to \eg the wave equation~\cite{HenningPalitta} and the Schroedinger equation~\cite{hain2022ultraweak}.
Unknown to the authors, this method already existed for many years in the least-squares community,
known as the $\LLstar$-method~\cite{caiFOSLL}. While exhibiting only suboptimal approximation qualities,
it allows for a reformulation of the original problem in terms of a symmetric and coercive dual problem
defined on the test space which can be thought of as the normal equations for the underlying
minimization problem. This is particularly convenient for model order reduction approaches as one can
then use classic techniques for symmetric problems.
Once solved, the dual solution is then used to reconstruct the primal solution.
In practice, one can often even abstain from an explicit reconstruction as most evaluations of the
solution can be reformulated in terms of functional evaluations of the dual solutions
(see~\cite{renelt2023} for details).
\par
The paper is now structured as follows: We first derive an ultraweak, optimally conditioned variational
formulation for Problem~\eqref{eq:introduction:strongTransportEquation} in Section~\ref{sec:transport}.
As the discretization and numerical solving of this particular problem have already been
discussed in~\cite{renelt2023}, we subsequently focus on the parametrized variant.
In Section~\ref{sec:approximability} we prove the exponential convergence of the
Kolmogorov $N$-width decay for the proposed parametrization which is based on the proof
for coercive problems~\cite{OhlbergerRave}. The prerequisites of our theorem allow for an abstract
investigation whether other parametrized inf-sup-stable problems are similarly suited for reduction.
Section~\ref{sec:modelOrderReduction} then continues with details on the generation of reduced spaces
for the dual variable. If such a reduced space has been obtained, we can employ an offline-online
splitting to minimize the computational costs for the reduced solving process. The selection of the
reduced basis functions is done by a weak Greedy algorithm utilizing a residual based error estimator.
Finally, we fix concrete test cases and present numerical results in Section~\ref{sec:numerics}. We
first verify the convergence of the full-order method and subsequently show the exponential
convergence of the approximation error of the reduced spaces produced by the greedy algorithm.
A comparison of the online runtimes shows the speedup gained by employing the reduced model.

\section{Ultraweak Petrov-Galerkin formulation of reactive transport}\label{sec:transport}
In this section we will introduce an ultraweak formulation of a reactive transport problem.
This results in the variational formulation proposed in~\cite{renelt2023} which we formally
derive and supplement by rigorous proofs.
The theory and notation are mostly based on~\cite{DahmenHuangSchwab}, which, in turn, employs
the concept of optimal test spaces - proposed in~\cite{demkowiczDPG1} - in a continuous setting.
\par
Let us first formally specify the strong formulation of the
parameter-independent reactive transport problem
\begin{equation}\label{eq:transport:strongProblem}
\text{Find $u$ such that} \qquad
\begin{cases}
\nabla\cdot(\vec{b} u) + cu &= f_\circ \phantom{g_D}\quad\text{in}\;\Omega,\\
\hfill u &= g_D \phantom{f_\circ}\quad\text{on}\;\Gin. \\
\end{cases}
\end{equation}
defined on an open and bounded domain $\Omega\subset\R^d$ with Lipschitz-boundary
$\Gamma := \partial\Omega$. The velocity field  $\vec{b}\in C^1(\Omega,\R^d)$ is assumed to be
divergence-free\footnote{We restrict ourselves to divergence free velocities. However,
non-divergence-free fields may be considered as well.} and $c\in \Linfty$, $c\geq 0$ a.e.
defines the reaction coefficient.
Moreover, $f_\circ\in\Ltwo$ denotes a source term and $g_D\in H^{1/2}(\Gin)$ the
boundary values at the inflow boundary, where the  in- and outflow boundary parts are defined as
$\Gamma_{out/in} := \{ z\in\Gamma \;|\; \vec{b}(z)\cdot \vec{n}(z) \gtrless 0\}$
and $\vec{n}(\cdot)$ denotes the outer unit normal. The induced variational form
\begin{equation}\label{eq:transport:strongVariationalForm}
  \text{Find}\; u\in \Xdense: \quad {(\nabla\cdot(\vec{b}u) + cu, v)_{\Ltwo}}
  = (f_\circ,v)_{\Ltwo} \qquad \forall v\in \Ydense.
\end{equation}
defined on the still highly regular vectorspaces
\begin{equation*}
\Xdense := \{u\in C^\infty(\Omega) \;|\; u=g_D \;\text{on}\;\Gin\}, \qquad \Ydense := C^\infty(\Omega),
\end{equation*}
now serves as the starting point of our derivation. Applying integration by parts we can write the
formal adjoint operator $A_\circ^*$ in the form
\begin{equation}\label{def:transport:adjointOperator}
  A_\circ^*[v](u) = (u, -\vec{b}\nabla v + cv)_{\Ltwo} + (u,v)_{\LtraceOut}.
\end{equation}
using the weighted inner product
${(u,v)}_{L^2(\Gamma_{out/in}, |\vec{b}\vec{n}|)} :=
\int_{\Gamma_{out/in}} u\,v\,|\vec{b}\vec{n}| \diff s$.
This also yields an additional linear term $(-g_D,v)_{\LtraceIn}$ on the right hand side to account for
the Dirichlet constraints.

We now need to equip $\Xdense$ and $\Ydense$ with suitable norms:
In the end we want to use a combination known as \textit{energy norm pairing} where the norm on
$\Xdense$ is the energy norm induced by the norm on $\Ydense$.
Instead of defining a norm on $\Ydense$ first, one can also define a norm on $\Xdense$
and then reconstruct the inducing norm on $\Ydense$, see~\cite{ChanHeuer}
for an in-depth discussion of this relation.

\subsection{Derivation and characterization of the trial space
  \texorpdfstring{$\xcal$}{X}}
In general, many choices of $\norm{\cdot}_{\Xdense}$ are possible leading to different variational formulations.
The only constraints we require are the following:
\begin{enumerate}[label={(N\arabic*)}, align=left, leftmargin=*]
\item For every $v\in \Ydense$, the operator $A_\circ^*[v]$ is continuous,
  \ie $A_\circ^* \in \mathcal{L}(\Ydense, \Xdense')$. \label{ass:continuity}
\item The norm $\norm{\cdot}_{\Xdense}$ is induced by a scalar product $(\cdot,\cdot)_{\Xdense}$.
\end{enumerate}
For the adjoint operator $A_\circ^*$ defined in~\eqref{def:transport:adjointOperator}
these assumptions are fulfilled if we choose
\begin{equation}\label{def:transport:xnorm}
  \norm{u}_{\Xdense}^2 := \Lnorm{u}^2 + \norm{u\restr{\Gout}}_{\LtraceOut}^2
\end{equation}
as the norm on $\Xdense$. Via closure we obtain the ultraweak trial space
\begin{equation*}
  \xcal := \clos_{\norm{\cdot}_{\Xdense}}(\Xdense)
\end{equation*}
equipped with the continuous extension $\norm{\cdot}_\xcal$
of the norm $\norm{\cdot}_{\Xdense}$.

\begin{remark}
  We want to emphasize again that~\eqref{def:transport:xnorm} is just one of many possible choices for
  $\norm{\cdot}_{\Xdense}$. One could for example also consider a mesh-dependent norm involving additional
  contributions on the skeleton which ultimately leads to a DPG-like formulation. Using a broken
  Sobolev-norm and thus enforcing more regularity in the trial space one can obtain a DPG*-like setting.
  Both of these formulations are, however, only efficient when combined with quasi-optimal norms
  which is not as suitable for reduction approaches.
\end{remark}

\begin{remark}
  Compared to the CPG~\cite{DahmenHuangSchwab}, we have dropped the (implicit) assumption
  that the adjoint differential operator $A_\circ^*$ can be seen as a mapping into $\Ltwo'$.
  By prescribing a slightly stronger norm on the trial
  functions, we have enlarged the dual space $\xcal'$ yielding a weaker continuity constraint
  on the adjoint operator. This then allows us to prescribe the boundary constraints on the test
  functions in a weak sense. As discussed in~\cite{BrunkenSmetanaUrban}, strong enforcement of the
  boundary conditions can lead to nonphysical artifacts and should thus be avoided.
\end{remark}

As it is known, the space of $C^\infty$-functions forms a dense subset of the Sobolev-spaces
$H^k(\Omega),\,k\geq 0$. Thus, the choice of a Sobolev-norm $\norm{u}_{\Xdense} := \norm{u}_{H^k(\Omega)}$
automatically yields $H^k(\Omega)$ as a trial space. However, as soon as either $\Xdense$ or the norm $\norm{\cdot}_{\Xdense}$ include constraints/contributions from the skeleton or boundary, one needs to pay attention
which space $\xcal$ one has actually constructed. We will thus shortly give a formal classification of
the space $\xcal$:

\begin{corollary}
  There exists a linear and continuous trace operator\footnote{We will abbreviate different trace
  operators by the same symbol $\gamma$ and only use subscripts if necessary.}
\begin{equation*}
\gamma_{\xcal}: \xcal \rightarrow \LtraceOut
\end{equation*}
as well as a linear and continuous projection operator
\begin{equation}
  \operatorname{pr}_{L^2}: \xcal \rightarrow \Ltwo
\end{equation}
fulfilling $\gamma_{\xcal}(u) = u\restr{\Gout}$ and $\operatorname{pr}_{L^2}(u) = u$ for all
$u\in \Xdense \subset \xcal$.
\end{corollary}
\begin{proof}
  Due to the choice of the trial space norm~\eqref{def:transport:xnorm} the operators are obviously
  continuous on $\Xdense$ and can therefore be continuously extended to all elements in $\xcal$.
\end{proof}

We will now show that these two operators fully characterize $\xcal$, \ie they are both surjective and
in combination also injective. Furthermore, the $\xcal$-norm can be computed from the norms of
$\operatorname{pr}_{L^2}(u)$ and $\gamma(u)$. First, we need to show that the values of these two norms
are independent of each other:
\begin{lemma}\label{lemma:diracSequence}
  For every $\hat{u}\in\LtraceOut$ there exists a sequence $(\varphi_i)_{i\in\N}\subset \Xdense$ such that
  \begin{equation*}
    \lim_{i\to\infty}\Lnorm{\varphi_i} = 0,
    \qquad \lim_{i\to\infty}\norm{\hat{u} - \varphi_i\restr{\Gout}}_{\LtraceOut} = 0
  \end{equation*}
\end{lemma}
\begin{proof}
  \newcommand{\HtraceOut}{H^{1/2}(\Gout)}
  We first extend $\hat{u}$ to the full boundary by setting
  $\hat{u}\restr{\Gin} := g_D$ and $\hat{u} := 0$ else.
  We then prove the statement assuming higher regularity $\hat{u}\in H^{1/2}(\partial\Omega)$;
  the general case follows by density. As the trace operator
  $\gamma: H^1(\Omega) \rightarrow H^{1/2}(\partial\Omega)$ is surjective
  and thus has a bounded right-inverse,
  we can define the element $u := \gamma^{-1}(\hat{u})\in H^1(\Omega)$.
  By density we now find sequences
  \begin{align*}
    (\psi^1_i)_{i\in\N}\subseteq \Xdense, &\quad \psi^1_i\xrightarrow{H^1(\Omega)} u, \\
    (\psi^2_j)_{j\in\N}\subseteq C_0^\infty(\Omega), &\quad \psi^2_j\xrightarrow{\Ltwo} u.
  \end{align*}
  Applying the trace operator $\gamma$ to the first limit and using its continuity one shows
  \begin{equation*}
    \gamma(\psi^1_i) \xrightarrow{H^{1/2}(\Gout)} \gamma(u) = \hat{u},
  \end{equation*}
  \ie the traces of $\psi^1_i$ converge to $\hat{u}$. Defining the sequence
  $\varphi_i := \psi^1_i - \psi^2_i \in \Xdense$, $i\in\N$, we observe
  \begin{align*}
    \Lnorm{\varphi_i}
    &\;\leq\; \hfill \Lnorm{\psi^1_i - u} + \Lnorm{u - \psi^2_i}
    \xrightarrow{i\to\infty} 0 \\
    \norm{\hat{u} - \varphi_i\restr{\Gout}}_{\HtraceOut}
    &\;=\; \hfill \norm{\hat{u} - \psi^1_i\restr{\Gout}}_{\HtraceOut}
    \xrightarrow{i\to\infty} 0.
  \end{align*}
\ie $(\varphi_i)_{i\in\N}$ has the desired properties.
\end{proof}

\newcommand{\range}{\operatorname{Im}}
\begin{proposition}[Characterization of $\xcal$]\label{thm:trialIsometry}
  The trial space $\xcal$ is isometrically isomorphic to the Sobolev-space $\xcal_{L^2}$, defined as
  \begin{equation*}
  \xcal_{L^2} := \Ltwo \times \LtraceOut.
  \end{equation*}
  equipped with the canonical product norm
  \begin{equation*}
    \norm{(u, \hat{u})}_{\xcal_{L^2}}^2 := \Lnorm{u}^2 + \norm{\hat{u}}_{\LtraceOut}^2.
  \end{equation*}
\end{proposition}
\begin{proof}
  Define the linear isomorphism
  \begin{equation*}
    \tilde{\Phi}: \Xdense \xrightarrow{1:1} \range\tilde{\Phi} \quad(\subseteq\xcal_{L^2}), \qquad \tilde{\Phi}(u) := (\mathrm{pr}_{L^2}(u), \gamma(u)).
  \end{equation*}
  which also is an isometry if we equip $\range\Phi$ with the norm $\norm{\cdot}_{\xcal_{L^2}}$.
  We can now continuously extend $\tilde{\Phi}$ to an isometry between the closures \ie
  \begin{equation*}
    \Phi: \xcal \xrightarrow{1:1} \clos_{\norm{\cdot}_{\xcal_{L^2}}}(\range\tilde{\Phi}),
    \qquad \Phi(\lim_{k\to\infty} u_k) := \lim_{k\to\infty} \Phi(u_k).
  \end{equation*}
  It remains to show that indeed
  \begin{equation*}
    \clos_{\norm{\cdot}_{\xcal_{L^2}}}(\range\tilde{\Phi}) \;=\; \xcal_{L^2}.
  \end{equation*}
  For arbitrary $(u,\hat{u})\in \xcal_{L^2}$ there exist sequences $(\varphi_k)_{k\in\N}, (\hat{\varphi}_l)_{l\in\N} \subset \Xdense$ such that
  \begin{align*}
    &\tilde{\Phi}(\varphi_k) \xrightarrow{\xcal_{L^2}} (u,0),
    & \hfill \text{($\{u\in \Xdense \,|\, u\restr{\Gout}=0\}$ is dense in $\Ltwo$)}\\
    &\tilde{\Phi}(\hat{\varphi}_l) \xrightarrow{\xcal_{L^2}} (0,\hat{u}),
    & \hfill \text{(Lemma~\ref{lemma:diracSequence})}
  \end{align*}
  and it follows that
  \begin{equation}
    \tilde{\Phi}(\varphi_k + \hat{\varphi}_k) \;\xrightarrow{\xcal_{L^2}}\; (u,\hat{u}).
    \tag*{\hspace*{1mm}}
  \end{equation}
\end{proof}

\begin{corollary}
  It holds that
\begin{align*}
  \norm{x}_\xcal^2 &=\quad \Lnorm{\mathrm{pr}_{L^2}(x)}^2
  + \norm{\gamma(x)}_{\LtraceOut}^2 = {(x,x)}_\xcal \\
  \text{with } \  (x,x')_\xcal &:=\quad {(\mathrm{pr}_{L^2}(x), \mathrm{pr}_{L^2}(x'))}_{\Ltwo}
  + {(\gamma(x), \gamma(x'))}_{\LtraceOut}.
\end{align*}
In particular, $\xcal$ is a Hilbert space and there exists the Riesz-map
$R_{\xcal}: \xcal \rightarrow \xcal'$.
\end{corollary}
In the following we will no longer distiguish between $\xcal$ and $\xcal_{L^2}$ and write
$(u, \hat{u})\in\xcal$ in the sense of  Prop.~\ref{thm:trialIsometry}.
Additionally, we continuously extend $A_\circ^*[v]$  to an operator acting on $\xcal$ and write
\begin{equation}\label{def:transport:adjointOperatorXcal}
  \forall v\in \Ydense:\quad A_\circ^*[v](u,\hat{u}) := {(u, -b\nabla v +cv)}_{\Ltwo}
  + {(\hat{u}, v\restr{\Gout})}_{\LtraceOut}.
\end{equation}

\subsection{The optimal test space
  \texorpdfstring{$\ycal$}{Y}}
We are now at a point similar to the setting of~\cite{DahmenHuangSchwab} replacing $\Ltwo$ with
the $L^2$-like space $\xcal$. Hence, we can proceed similarly if the following
assumptions hold:
\begin{enumerate}[label={(A\arabic*)}, align=left, leftmargin=*]
\item The adjoint operator $A_\circ^*: \Ydense \rightarrow \xcal'$ is injective.\label{ass:injectivity}
\item The range of the adjoint operator $\operatorname{rg}(A_\circ^*)$ is dense in $\xcal'$.
  \label{ass:surjectivity}
\end{enumerate}
The following proposition now gives a sufficient condition on the data functions for the
assumptions \ref{ass:injectivity} and \ref{ass:surjectivity}. This requires the notion of
an $\Omega$-filling transport field:
\begin{definition}[$\Omega$-filling flow]
  For a flow field $\vec{b}$ let $\xi(t,x)$ define the integral curves satisfying
  \begin{equation*}
    (\partial_t\xi)(t,x) = \vec{b}(\xi(t,x)), \quad \xi(0,x) = x.
  \end{equation*}
  A flow field $\vec{b}$ is called $\Omega$-filling if there exists a function
  $T\in\Linfty$ such that for almost all $x\in\Omega$ there exists
  a $x_\Gamma\in\Gin$ with $\xi(T(x),x_\Gamma) = x$. We call $T_{max}\;:=\norm{T}_{\Linfty}$ the
  \textit{maximal traverse time}. We additionally require that $T$ is defined
  on the outflow boundary (in the sense of traces)
  with $\norm{T}_{L^\infty(\Gout)} = \norm{T}_{\Linfty} = T_{max}$.
\end{definition}
\begin{definition}[Minimal traverse time]
  Let $\vec{b}$ be $\Omega$-filling. We define the minimal traverse time as
  \begin{equation}\label{def:traverseTime}
    T_{min}:=\operatorname{essinf}_{z\in\Gout}\{T(z)\}.
  \end{equation}
  Note, that we have $0 \leq T_{min} < \infty.$
\end{definition}
\begin{proposition}\label{prop:transport:poincare}
  Let one of the following conditions hold:
  \begin{enumerate}
  \item The transport field $\vec{b}$ is $\Omega$-filling.
  \item The reaction coefficient $c$ is bounded from below, \ie~$c(x) \geq \kappa > 0$.
  \end{enumerate}
  Then, the adjoint operator $A_\circ^*:\Ydense\rightarrow \xcal'$~\eqref{def:transport:adjointOperatorXcal}
  fulfills the assumptions \ref{ass:injectivity} and \ref{ass:surjectivity}.
  Moreover, the Poincaré-type inequality
  \begin{equation}\label{eq:transportPoincare}
    \norm{v}_\xcal \leq C_p\norm{A_\circ^*[v]}_{\xcal'}
  \end{equation}
  holds for every $v\in \Ydense$. The Poincaré constant $C_p$ can be bounded as follows:
  \begin{equation*}
    C_p \leq (2\,T_{max}+\max\{1-T_{min},0\})
  \end{equation*}
  or, if the second condition holds,
  \begin{equation*}
    C_p \leq \max\{2,\kappa^{-1}\}.
  \end{equation*}
\end{proposition}
\begin{proof}
  see Appendix~\ref{appendix:adjointTransportOperator}
\end{proof}
Following, and in particular in Sec.~\ref{sec:numerics}, we will assume the flow field to be
$\Omega$-filling.\par
Using the characterization~\eqref{def:transport:adjointOperatorXcal} we see that for a given
$v\in \Ydense$ its Riesz-re\-pre\-sen\-ta\-tive $r_v := R^{-1}_{\xcal}(A_\circ^*[v])\in\xcal$ is given as
$r_v = (-b\nabla v + cv, v\restr{\Gout})$. We can thus identify the abstractly defined test space norm
\begin{equation*}
  \norm{v}_{\Ydense} := \norm{A_\circ^*[v]}_{\xcal'} = \norm{R^{-1}_{\xcal}(A_\circ^*[v])}_\xcal,
\end{equation*}
which is indeed a norm due to \ref{ass:injectivity}, with the representation
\begin{equation*}
\norm{v}_{\Ydense}^2 = \norm{r_v}_\xcal^2 = \Lnorm{-b\nabla v + cv}^2 + \norm{v\restr{\Gout}}_{\LtraceOut}^2.
\end{equation*}
The test space $\ycal$ is then defined by
\begin{equation*}
  \ycal := \clos_{\norm{\cdot}_{\Ydense}}(\Ydense)
\end{equation*}
equipped with the induced norm
\begin{equation}
\norm{y}_\ycal := \norm{A^*[y]}_{\xcal'}
\end{equation}
where $A^*:\ycal \rightarrow \xcal'$ denotes the continuous extension of $A_\circ^*$ to $\ycal$.

In the next section (Sec.~\ref{ssec:optimallyStableUWFormulation}) we will show that this space is 'optimal' in the sense that the variational formulation using the trial/test-pair $(\xcal,\ycal)$ has optimal condition number of one. In the remainder of this section we also want to characterize $\ycal$ in terms of a more accessible Sobolev-like space:
\begin{remark}
  One easily sees that the trace $\gamma(v) := v\restr{\Gout}$ for $v\in \Ydense$ is continuous and can thus
  be extended to a trace operator $\gamma: \ycal\rightarrow\LtraceOut$ with operator norm
  $\norm{\gamma}\leq 1$:
  \begin{equation*}
    \norm{v}_{\LtraceOut}^2 = |A_\circ^*[v](0,v\restr{\Gout})| \leq \norm{v}_\ycal\norm{v}_{\LtraceOut}
    \qquad\forall v\in \Ydense.
  \end{equation*}
\end{remark}
Functions in $\ycal$ also have a trace in $\LtraceIn$ as the following lemma shows:
\begin{lemma}\label{lemma:traceTheorem}
  There exists a linear and continuous trace operator
  \begin{equation*}
    \gamma: \ycal \rightarrow \LtraceIn, \qquad \norm{\gamma(v)}_{\LtraceIn}
    \leq \sqrt{2C_p}\norm{v}_{\ycal}
  \end{equation*}
  with $\gamma(v) = v\restr{\Gin}$ for all $v\in \Ydense\subset\ycal$.
\end{lemma}
\begin{proof}
  For arbitrary $v\in \Ydense$, integration by parts yields
  \begin{equation*}
    A_\circ^*[v](v,v\restr{\Gout})
    = \tfrac{1}{2}\norm{v}_{\LtraceOut}^2 + \tfrac{1}{2}\norm{v}_{\LtraceIn}^2
    + \underbrace{\int_\Omega cv^2 \diff x}_{\geq 0}
    \geq \tfrac{1}{2}\norm{v}_{\LtraceIn}^2.
  \end{equation*}
  Using the definition of the dual-norm and the Poincaré-inequality \eqref{eq:transportPoincare}
  we obtain
  \begin{equation*}
    \norm{v}_{\LtraceIn}^2 \leq 2|A_\circ^*[v](v,v\restr{\Gout})|
    \leq 2\norm{A_\circ^*[v]}_{\xcal'}\norm{v}_\xcal \leq 2C_p\norm{A^*_\circ[v]}_{\xcal'}^2.
  \end{equation*}
  The statement for all $v\in\ycal$ again follows by density.
\end{proof}

\begin{definition}[Sobolev-space $\Honeb$]
  Define the following norm $\norm{\cdot}_{\Honeb}$ for $v\in \Ydense=C^\infty(\Omega)$:
  \begin{equation*}
    \norm{v}_{\Honeb}^2 := \Lnorm{\vec{b}\nabla v}^2 + \Lnorm{v}^2.
  \end{equation*}
  Then, the space $\Honeb$ is defined as the closure of $\Ydense$ under the norm.
\end{definition}

\begin{lemma}\label{lemma:normEquivalence}
  Let $\vec{b}$ be $\Omega$-filling and the minimal traverse time be bounded away from zero,
  \ie $T_{min}>0$. Then the norm equivalence
  \begin{equation*}
    c\norm{v}_{\Honeb} \quad\leq\quad \norm{v}_{\ycal} \quad\leq\quad C\norm{v}_{\Honeb}
  \end{equation*}
  holds with equivalence constants $c,\,C$ given by
  \begin{align*}
    c &= (2 + C_p^2(2\norm{c}_{\Linfty}^2+1))^{-\tfrac{1}{2}} \\
    C &= (\max \{2\norm{c}_{\Linfty}, T_{min}^{-1}(2T_{max}+1)\})^{\tfrac{1}{2}}
  \end{align*}
  where $C_p$ denotes the Poincaré-type constant from~\eqref{eq:transportPoincare}.
\end{lemma}
\begin{proof}
  The proof is similar to~\cite[Prop.~2.6]{DahmenHuangSchwab}, however we obtain slightly different
  constants due to the additional boundary norm and the usage of the
  Poincaré-type estimate~\eqref{eq:transportPoincare}.\\
  In the following let $v\in \Ydense$ (the statement for all $v\in\ycal$ follows by density).
  \begin{itemize}
  \item Using the triangle inequality and Youngs theorem yields
    \begin{align*}
      \norm{v}_{\ycal}^2
      &= \Lnorm{-\vec{b}\nabla v + cv}^2 + \norm{v}_{\LtraceOut}^2 \\
      &\leq 2\Lnorm{-\vec{b}\nabla v}^2
      + 2\norm{c}_{\Linfty}^2 \, \Lnorm{v}^2 + \norm{v}_{\LtraceOut}^2 \\
      &\leq \max\{2, 2\norm{c}_{\Linfty}^2, T_{min}^{-1}(2T_{max}+1)\}\norm{v}_{\Honeb}^2 \\
      &= \underbrace{\max\{2\norm{c}_{\Linfty}^2, T_{min}^{-1}(2T_{max}+1)\}}_{=:C^2}
      \norm{v}_{\Honeb}^2.
    \end{align*}
    where in the last two steps we used a trace theorem for $\Honeb$
    (see Appendix~\ref{appendix:traceTheoremsH1b}) and the fact that $T_{min}\leq T_{max}$.
  \item To prove the lower bound we expand the streamline term in the $\Honeb$-norm
    and use Youngs theorem:
    \begin{align*}
      \norm{v}_{\Honeb}^2
      &\leq (\Lnorm{-\vec{b}\nabla v + c v} + \norm{c}_{\Linfty}\Lnorm{v})^2 \\
      & \qquad + \Lnorm{v}^2 + \norm{v}_{\LtraceOut}^2\\
      &\leq 2\norm{v}_{\ycal}^2 + (2\norm{c}_{\Linfty}^2 + 1)\Lnorm{v}^2.
    \end{align*}
    Using the norm inequality~\eqref{eq:transportPoincare} then yields the assertion.
  \end{itemize}
\end{proof}

\begin{corollary}\label{prop:testSpaceIdentification}
  Under the assumptions in Lemma~\ref{lemma:normEquivalence} the test space $\ycal$ is isomorphic
  to the Sobolev-space $\Honeb$.
\end{corollary}

We briefly want to discuss the assumptions we made in Lemma~\ref{lemma:normEquivalence}:

\begin{remark}
  The identification of $\ycal$ with $\Honeb$ is valid under the assumption of an $\Omega$-filling flow
  and a positive minimal traverse time. While the first assumption is quite natural there are very
  simple examples where the traverse time is not bounded from below, even for domains with
  $C^\infty$-boundary (consider for example the constant velocity field in $x$-direction
  $\vec{b} \equiv e_1$ defined on the unit disc). However, even in those cases we still have
  the following chain of continuous embeddings:
  \begin{equation*}
    H^1(\Omega) \;\hookrightarrow\; \ycal \;\hookrightarrow\; \Honeb.
  \end{equation*}
\end{remark}

\subsection{Optimally stable ultraweak formulation and normal equation}
\label{ssec:optimallyStableUWFormulation}
The constructed test space $\ycal$ is optimal in the sense that it contains all the supremizers
of $\xcal$. Therefore, the resulting variational formulation can be shown to be
\textit{optimally stable}, \ie has continuity- and inf-sup-constant of one:

\begin{proposition}[{\cite[Prop.\ 2.1]{DahmenHuangSchwab}}]\label{prop:genericIsometry}
  Let the assumptions \ref{ass:injectivity} and \ref{ass:surjectivity} hold. Then, the mappings
  $A: \xcal\rightarrow\ycal'$ and $A^*:\ycal\rightarrow \xcal'$ are isometries, \ie,
  \begin{equation*}
    \ycal = A^{-*}\xcal', \quad \xcal = A^{-1}\ycal'
  \end{equation*}
  and
  \begin{equation*}
    \norm{A}_{\mathcal{L}(\xcal,\ycal')} = \norm{A^*}_{\mathcal{L}(\ycal,\xcal')}
    = \norm{A^{-1}}_{\mathcal{L}(\ycal',\xcal)} = \norm{A^{-*}}_{\mathcal{L}(\xcal',\ycal)}
    = 1.
  \end{equation*}
\end{proposition}

\begin{corollary}[{\cite[Prop.\ 2.1]{DahmenHuangSchwab}}]\label{cor:optimalCondition}
The variational formulation
\begin{equation}\label{eq:transport:continuousFormulation}
  \text{Find}\;u\in\xcal: \qquad a(u,v) = f(v) \qquad \forall v\in\ycal
\end{equation}
with $a(u,v) :=$ $(Au,v)_{\ycal'\times\ycal} =$ $(u,A^*v)_{\xcal\times\xcal'}$ and $f\in\ycal'$ is well-posed and
has condition number $\kappa_{\xcal,\ycal}(A) := \norm{A}\,\norm{A^{-1}} = 1$.
\end{corollary}

If we use the particular right-hand side
$f(v) := {(f_\circ, v)}_{\Ltwo} - {(g_D, v)}_{\LtraceIn}$
we obtain the ultraweak formulation introduced in~\cite[Def.1]{renelt2023}. Formally, it remains
to verify that $f$ is indeed an element of $\ycal'$ which, however, directly follows from
Lemma~\ref{lemma:traceTheorem} and \eqref{eq:transportPoincare}.
\begin{corollary}[Recovery of strong solutions]
  If the ultraweak solution $(u,\hat{u})\in\xcal$ is sufficiently regular,
  \ie~$u\in C^1(\overline{\Omega})$, then $(u,\hat{u})$ also solve the strong formulation
  \begin{equation}\label{eq:transport:strongFormulation}
    A_\circ u = f_\circ \quad\text{in}\;\Omega,
    \qquad  u = g_D \quad\text{on}\;\Gin,
    \qquad \hat{u} = u \quad\text{on}\;\Gout.
  \end{equation}
\end{corollary}

Due to Proposition~\ref{prop:genericIsometry} for every $u\in\xcal$ there exists a unique $w\in\ycal$
with $u = R^{-1}_{\xcal} A^*w$. This allows us to write
Problem~\eqref{eq:transport:continuousFormulation} in the following mixed form:
\begin{equation}\label{eq:generalSaddlepointProblem}
\text{Find}\; u\in\xcal,\; w\in\ycal: \qquad
\begin{cases}
R_\xcal u + (-A^*)w &= 0 \qquad \text{in } \xcal', \\
Au  &= f \qquad \text{in } \ycal'. \\
\end{cases}
\end{equation}
The second equation is the ultraweak formulation of our PDE while the first equation forces $u$ to be an
element of the range of $R^{-1}_{\xcal} A^*$ which nicely shows the saddlepoint structure of the
problem. By eliminating $u$ we obtain the symmetric problem
\begin{equation}
  \text{Find}\; w\in\ycal: \qquad AR^{-1}_{\xcal} A^*w = f \qquad \text{in}\; \ycal'.
\end{equation}
or equivalently
\begin{equation}\label{eq:transport:continuousNormalEq}
\text{Find}\; w\in\ycal: \qquad \hat{a}(w,v) = f(v) \qquad \forall v\in\ycal.
\end{equation}
with $\hat{a}(w,v) := (A^* w, A^*v)_{\xcal'}$
which we will call the \textit{normal equation} associated with
Problem~\eqref{eq:transport:continuousFormulation}.
Note, that due to Prop.~\ref{prop:genericIsometry} this
is in fact an equivalent formulation of the `classic' normal equation
\begin{equation}\label{eq:generalNormalEq}
  \text{Find}\;u\in\xcal:\quad A^*R^{-1}_{\ycal}Au = A^*R^{-1}_{\ycal}f \qquad\text{in}\;\xcal'.
\end{equation}
or (in variational form)
\begin{equation} \label{eq:generalSymmetricProblem2}
\text{Find}\; u\in\xcal: \qquad (Au,Av)_{\ycal'} = (f,Av)_{\ycal'} \qquad \forall v\in\xcal.
\end{equation}
In this sense our approach can also be interpreted as an (FOS)$\LLstar$-method~\cite{caiFOSLL} for minimizing the
residual in the \textit{dual norm}, \ie
\begin{equation*}
u \;=\; \operatorname{argmin}_{\tilde{u}\in\xcal}\norm{A\tilde{u}-f}_{\ycal'}.
\end{equation*}
Since we use the optimal test space norm and thus have a closed form of the Riesz-representative
$R^{-1}_{\xcal} A^*[v] = (-\vec{b}\nabla v + cv, \gamma_{out}(v))$ we can also explicitly write
the normal equation~\eqref{eq:transport:continuousNormalEq} as
\begin{equation}\label{eq:transport:explicitContinuousNormalEq}
 {(-b\nabla w + cw, -b\nabla v + cv)}_{\Ltwo} + {(\gamma(w),\gamma(v))}_{\LtraceOut} \quad =\quad f(v).
\end{equation}

\section{Discretization}\label{sec:discretization}
For the subsequent numerical experiments we discretize the normal
equation~\eqref{eq:transport:explicitContinuousNormalEq} using Galerkin-projection and obtain the
discrete solution $w^\delta$ as an element of a finite dimensional subspace
$\ycal^\delta \subseteq \ycal$ solving
\begin{equation}\label{eq:discreteNormalEq}
\text{Find}\; w^\delta\in\ycal^\delta: \qquad \hat{a}(w^\delta,v^\delta) = f(v^\delta) \qquad\forall v^\delta\in\ycal^\delta.
\end{equation}
As discussed in~\cite{renelt2023} this only requires a discretization of the optimal test space
$\ycal$ since the associated trial space $\xcal = R^{-1}_{\xcal} A^*[\ycal]$ is just required for
the reconstruction. However, we never actually use the full discrete solution
$u^\delta := R^{-1}_{\xcal}A^*[w^\delta]$ but only functional evaluations of
$w^\delta$ and can thus omit a
discretization of $\xcal$. Note, that the projected normal equation~\eqref{eq:discreteNormalEq}
automatically inherits the optimal continuity and coercivity.
Considering the Petrov-Galerkin formulation
\begin{equation}\label{eq:discretePetrovGalerkinEq}
  \text{Find}\; u^\delta\in\xcal^\delta: \qquad (u^\delta, A^*[v^\delta])_{\xcal\times\xcal'} = f(v^\delta) \qquad\forall v^\delta\in\ycal^\delta.
\end{equation}
we see that this formulation is optimally stable for the specific choice of the trial space
$\xcal^\delta := R^{-1}_{\xcal} A^*[\ycal^\delta]$.

\begin{remark}\label{remark:quasiOptimalNormalEq}
  The normal equation~\eqref{eq:transport:continuousNormalEq} can also be formulated for any other
  norm on $\ycal$ equivalent to the optimal norm. However, in most cases it is then no longer possible
  to obtain a closed form of the inverse Riesz-map (and thus of the normal equations) and one has to
  additionally discretize the trial space $\xcal$.
\end{remark}

\subsection{A-priori error analysis}
By applying Céa's lemma to~\eqref{eq:discreteNormalEq} we obtain that the adjoint solution
$w^\delta$ is the best-approximation in the operator-dependent norm $\norm{\cdot}_\ycal$, \ie
\begin{equation}\label{eq:bestApproximationTestspace}
  \norm{w - w^\delta}_{\ycal} = \inf_{\tilde{w}\in\ycal^\delta}\norm{w - \tilde{w}}_{\ycal}.
\end{equation}
This is exactly equivalent to the fact that the reconstruction $u^\delta$ is the
$\xcal$-best-ap\-pro\-xi\-ma\-tion in the non-standard discrete space
$\xcal^\delta = R^{-1}_{\xcal} A^*[\ycal^\delta]$, \ie
\begin{equation}\label{eq:bestApproximationReconstruction}
  \norm{u - u^\delta}_{\xcal} = \inf_{\tilde{u}\in\xcal^\delta}\norm{u - \tilde{u}}_{\xcal}
\end{equation}
where $u = R^{-1}_{\xcal} A^*[w]\in\xcal$ denotes the continuous solution
to~\eqref{eq:transport:continuousFormulation}. This illustrates one of the drawbacks of the
optimal trial method and related methods (such as $\LLstar$- or DPG*-methods):
Classic interpolation arguments to deduce convergence rates are generally only possible
in formulation~\eqref{eq:bestApproximationTestspace} since $\ycal^\delta$ is the discrete space
we actually prescribe and an interpolation operator $I: \xcal \rightarrow \xcal^\delta$ with
useful properties is typically difficult to construct. Therefore, the rate of convergence is
limited by the regularity of the continuous adjoint solution $w$. Even for smooth $u$,
$w$ might only be smooth in the interior as discussed in~\cite{keithAPriori}.

\begin{remark}\label{rem:transport:convergenceRates}
  In the case of linear advection, we expect the following convergence rates:
  Assuming that the continuous adjoint solution $w$ is regular enough, we can use standard a-priori
  results for the Lagrange-space $\ycal^\delta := \mathbb{Q}^k(\mathcal{T}_h)$ to obtain the
  worst-case bound
  \begin{equation*}
    \norm{u - u^\delta}_{\xcal} = \norm{w - w^\delta}_{\ycal}
    \leq \inf_{\tilde{w}\in\ycal^\delta}\norm{w - \tilde{w}}_{H^1(\Omega)}
    \leq C|w|_{H^{s+1}(\Omega)}|h|^{s}.
  \end{equation*}
  for $s\in\{1,...,k\}$, \ie at least $k$-th order convergence.
  Higher rates can only be expected in special cases.
\end{remark}

For detailed analysis and discussion we refer to~\cite{keithAPriori,demkowiczDPGstar}.
Experimental convergence rates for transport problems with spatially constant advection
and inflow conditions of varying regularity have been computed in~\cite{BrunkenSmetanaUrban}.
In Section~\ref{sec:numerics} we will complement these results with experimental data for
a more complex model problem motivated by the reactive transport in a catalytic filter.

\section{Approximability of the parametrized problem}\label{sec:approximability}
\subsection{The parametrized problem}
From now on we consider the advection-reaction problem
\begin{equation}\label{eq:parametrizedStrongProblem}
  \text{Find}\; u_\mu\in\xcal: \qquad
  \begin{cases}
    \nabla\cdot(\vec{b} u_\mu) + c_\mu u_\mu &= f_{\circ,\mu} \qquad\text{in}\;\Omega, \\
    \hfill u_\mu &= g_{D,\mu} \qquad
    \text{on}\;\Gin
  \end{cases}
\end{equation}
with parameter-dependent data functions $c_\mu$, $f_\mu$ and $g_{D,\mu}$.
We assume that the parameter $\mu$ lies in a compact set $\Pcal\subset\R^p$.
Note, that the velocity field is not parameter-dependent.
For a fixed $\mu$ we can now apply the variational setting detailed in
Sec.~\ref{sec:transport} and obtain the problem
\begin{equation}\label{eq:parametrizedVariationalProblem}
  \text{Find}\; u_\mu\in\xcal: \quad {(u_\mu,A_\mu^*[v])}_{\xcal \times \xcal'} = f_\mu(v)
  \qquad \forall v\in\ycal_\mu.
\end{equation}
which is well-posed and optimally stable. Due to the parameter-dependency in the adjoint operator
$A_\mu^*$ the test space norm $\norm{v}_\mu := \norm{A_\mu^*[v]}_{\xcal'}$ and the resulting
test space $\ycal_\mu := \clos_{\norm{\cdot}_\mu}(\Ydense)$ differ for each parameter.

\subsection{Approximation theory}
Based on problem~\eqref{eq:parametrizedVariationalProblem} we define the set of all solutions
\begin{equation*}
  \Mcal := \{u_\mu \;|\; \mu\in\Pcal,\;
  u_\mu \;\text{solves}\;\eqref{eq:parametrizedVariationalProblem} \} \quad \subseteq \xcal,
\end{equation*}
which we call the `solution manifold'.
The core idea of model order reduction is to approximate this set by low-dimensional
linear subspaces $X_N \subset \xcal$. Therefore, it is crucial that the Kolmogorov $N$-width
\begin{equation}\label{eq:kolmogorovNWidth}
  d_N(\Mcal) := \Inf_{\dim(U_N)=N}\; \sup_{u\in\Mcal}\; \Inf_{u_N\in U_N}, \norm{u-u_N}_\xcal
\end{equation}
which describes the approximation error made by the best possible $N$-dimensional linear
subspaces, decays quickly. Although the optimal space will rarely be computable we can still hope
to construct \textit{quasi-optimal spaces} whose approximation error decays with a similar rate.
It is known that for coercive or inf-sup-stable problems one can expect exponential decay
of $d_N(\Mcal)$ under a few additional assumptions such as
parameter-separability~\cite{OhlbergerRave,urban2023reduced}.
In order to apply these results in the context of ultraweak
formulations we need to show that the parameter-dependent spaces
$\ycal_\mu$ are 'similar enough' which we formalize as follows:

\begin{proposition}\label{prop:kolmogorovUltraweak}
  Let $\Mcal\subseteq\xcal$ be the solution manifold for a (general) parametrized ultraweak problem
  in the form
  \begin{equation}\label{eq:approximability:generalProblem}
    \text{Find}\;u_\mu\in \xcal:\quad (u_\mu,A_\mu^*[v])_{\xcal\times\xcal'}
    \;=\; f_\mu(v) \qquad \forall v\in\ycal_\mu.
  \end{equation}
  Denote by $\hat{a}_\mu(w,v) := (A^*_\mu [w],A^*_\mu [v])_{\xcal'}$ the bilinear form of the associated normal equation.
  Assume that the following holds:
  \begin{enumerate}
  \item The parameter-independent space $\ycal_0 := \cap_{\mu\in\Pcal} \ycal_\mu$
    is dense in every $\ycal_\mu$.
  \item $\ycal_0$ can be equipped with a norm $\norm{\cdot}_0$ equivalent to the optimal test norm
    $\norm{v}_\mu = \norm{A_\mu^*[v]}_{\xcal'}$, \ie there exist $\mu$-independent equivalence
    constants $C,c > 0$ such that
    \begin{equation}\label{eq:approximability:normEquivalence}
      c\norm{v}_0 \;\leq\; \norm{v}_\mu \;\leq\; C\norm{v}_0.
    \end{equation}
  \item $\hat{a}_\mu$ and $f_\mu$ are parameter-separable on $\ycal_0$,
  \ie there exist continuous functions
    $\theta^b_i, \theta^f_i: \Pcal \rightarrow \R$ and (bi-)linear and continuous functionals
    $\hat{a}_i: \ycal_0\times\ycal_0 \rightarrow \R$, $f_i: \ycal_0\rightarrow\R$ such that
  \begin{equation*}
    \forall v,w\in\ycal_0: \quad \hat{a}_\mu(w,v) \;=\; \sum_{i=1}^{Q_a}\theta^a_i(\mu)\hat{a}_i(u,v),
    \quad f_\mu(v) \;=\; \sum_{i=1}^{Q_f}\theta^f_i(\mu)f_i(v).
  \end{equation*}
  \end{enumerate}
  Then, the Kolomogorov $N$-width $\mathrm{d}_N(\Mcal)$ decays exponentially, \ie
  \begin{equation*}
    \mathrm{d}_N(\Mcal) \leq \alpha\cdot \exp(-\beta N^{1/Q_a}).
  \end{equation*}
  For a parametrized righthand side $f_\mu$ we have $\alpha\in\mathcal{O}(Q_f)$ and
  $\beta\in\mathcal{O}(Q_f^{-1/Q_a})$.
\end{proposition}
\begin{proof}
  \newcommand{\ybar}{\bar{\ycal}_0}
  Define the parameter-independent space $\ybar := \clos_{\norm{\cdot}_0}(\ycal_0)$
  and the normal equation
  \begin{equation*}
    \text{Find}\; w\in\ybar:\quad \hat{a}_\mu(w, v) = f_\mu(v)
    \qquad \forall v\in\ybar.
  \end{equation*}
  with solution manifold $\Mcal_a\subseteq\ybar$. Consider the subproblems
  \begin{equation*}
    \text{Find}\; w_i\in\ybar: \quad \hat{a}_\mu(w_i, v) = f_i(v)
    \qquad \forall v\in\ybar, \qquad i=1,...,Q_f.
  \end{equation*}
  with corresponding solution manifold $\Mcal_i\subseteq\ybar$.
  If we can bound $\mathrm{d}_N(\Mcal_i)$
  from above we obtain a (not necessarily sharp) upper bound on $\mathrm{d}_N(\Mcal_a)$ via
  \begin{equation*}
    \mathrm{d}_{(Q_f N)}(\Mcal_a) \leq \sum_{i=1}^{Q_f}\mathrm{d}_N(\Mcal_i).
  \end{equation*}
  Following, we thus assume $f$ to be parameter-independent. Consider the normal equation
  \begin{equation}\label{eq:approximability:normalEq}
    \text{Find}\; w\in\ybar:\quad \hat{a}_\mu(w, v) = f(v) \qquad \forall v\in\ybar.
  \end{equation}
  Due to~\eqref{eq:approximability:normEquivalence} the bilinear form $\hat{a}_\mu$ is coercive
  and continuous on $\bar{\ycal}_0$. Applying theory for symmetric and coercive
  problems~\cite{OhlbergerRave} yields that $\mathrm{d}_N(\Mcal_a)$ decays with the proposed rate.
  Using the norm equivalence~\eqref{eq:approximability:normEquivalence}
  and the denseness of $\bar{\ycal}_0$ in all $\ycal_\mu$ yields
  $\mathrm{d}_N(\Mcal) \leq \mathrm{d}_N(\Mcal_a)$ and thus the claim.
\end{proof}

\begin{remark}
  In the case of reactive transport we have seen in Cor.~\ref{prop:testSpaceIdentification}
  that every $\mu$-dependent test space $\ycal_\mu$
  is isomorphic to the parameter-independent Sobolev-space $\Honeb$. Moreover, the norm equivalence
  \begin{equation*}
    c(\mu)\norm{v}_{\Honeb} \;\leq\; \norm{v}_{\mu} \;\leq\; C(\mu)\norm{v}_{\Honeb}.
  \end{equation*}
  holds with parameter-dependent equivalence constants
  \begin{align*}
    c(\mu) &= (2 + C_p^2(2\norm{c_\mu}_{\Linfty}^2+1))^{-\tfrac{1}{2}} \;> 0,\\
    C(\mu) &= (\max \{2\norm{c}_{\Linfty}, T_{min}^{-1}(2T_{max}+1)\})^{\tfrac{1}{2}} \;< \infty.
  \end{align*}
  Due to the parameter-separability (which one directly verifies) these constants continuously depend on $\mu$ and can
  thus be uniformly bounded from below or above, respectively.
\end{remark}

\begin{remark}
  For a parametrized transport direction $\vec{b}_\mu$, the test spaces $\ycal_\mu$ are isomorphic
  to the spaces $H^1(\vec{b}_\mu, \Omega)$ which differ for space-dimension $d>1$. Thus, we can,
  as it is known in the community (see \eg\cite{BrunkenSmetanaUrban,OhlbergerRave}),
  no longer expect exponential convergence of $d_N(\Mcal)$.
\end{remark}

\section{Test space based model order reduction}\label{sec:modelOrderReduction}
As we have seen in the discretization of the non-parametric problem (Sec.~\ref{sec:discretization})
it is advisable to perform as many computations as possible in the test space. Following this idea,
we use a reduced formulation of the adjoint normal
equation~\eqref{eq:transport:explicitContinuousNormalEq}
(as proposed in~\cite{BrunkenSmetanaUrban} and also used in~\cite{HenningPalitta}). Let $\ycal^\delta$ be a finite dimensional discretization of $\Honeb$ with dimension $\text{dim}(\ycal^\delta) = n$ obtained \eg by a high-dimensional finite element discretization and $w^\delta\in\ycal^\delta$ the corresponding high-dimensional adjoint solution of~\eqref{eq:discreteNormalEq}. For a small linear subspace $Y^N\subseteq\ycal^\delta$, $\text{dim}(Y^N) = N \ll n$, a reduced solution $w_N \in Y^N$ is given by Galerkin-projection, \ie by solving the reduced normal equation
\begin{equation}\label{eq:reducedNormalEq}
  \text{Find}\; w_N\in Y^N: \qquad (A_\mu^*[w_N], A_\mu^*[v_N])_{\xcal'} \;=\; f_\mu(v_N) \qquad\forall v_N\in Y^N.
\end{equation}
Note, that the reduced problem is coercive with optimal coercivity constant $\alpha(\mu) = 1$ when equiping $Y^N$ with the optimal norm $\norm{\cdot}_\mu$. Equivalently, the reduced Petrov-Galerkin formulation\begin{equation}\label{eq:reducedEq}
  \text{Find}\; u_N\in X^N: \qquad (u_N, A_\mu^*[v_N])_{\xcal\times\xcal'} \;=\; f_\mu(v_N) \qquad\forall v_N\in Y^N.
\end{equation}
is optimally stable if we choose $X^N := R^{-1}_{\xcal} A^*[Y^N]$.

\subsection*{Offline-online decomposition}
As noted before, the operators $A_\mu^*$ and $f_\mu$ in the reactive transport problem
are parameter-separable, \ie
\begin{equation*}
A_\mu^*[w] = \sum_{q=1}^{Q_a} \theta_q^a(\mu)\,A_q^*[w] \quad\text{and}\quad f_\mu(v) = \sum_{q=1}^{Q_f} \theta_q^f(\mu)\,f_q(v)
\end{equation*}
with continuous coefficient functions $\theta_q^a,\, \theta_q^b$ and continuous linear functionals $A^*_q,\, f_q$. Let $\{w_i^N\}_{i=1}^N$ be a basis of $Y^N$. The following quantities can then be computed in the offline-phase, \ie independent of a concrete parameter:
\begin{align*}
Y_{i,j}^{p,q} &:= (A_p^*w_i^N,A_q^*w_j^N)_{\xcal'} & \text{for }p,q=1,...,Q_a \quad\text{and }\; i,j=1,...,N \, ,\\
f_i^q &:= f_q(w_i^N) & \text{for } q=1,...,Q_f \quad\text{and }\; i=1,...,N .
\end{align*}
For a given parameter $\mu\in\Pcal$, assembling the system matrix of~\eqref{eq:reducedNormalEq} then only requires the computations
\begin{equation*}
\underline{Y}_\mu^\mathrm{red} = \sum_{p=1}^{Q_a}\sum_{q=1}^{Q_a}\theta_p^a(\mu)\theta_q^a(\mu)Y^{p,q} , \qquad
\underline{f}_\mu^\mathrm{red} = \sum_{q=1}^{Q_f}\theta_q^f(\mu)f^q,
\end{equation*}
with computational complexity $\mathcal{O}(Q_a N^2)$ or $\mathcal{O}(Q_f N)$, respectively. The corresponding linear equation system $\underline{Y}^{red}_\mu \underline{w}_\mu = \underline{f}^{red}_\mu$ is dense and can be solved in $\mathcal{O}(N^3)$.

\begin{remark}[Condition of the reduced system]
As discussed in~\cite{renelt2023}, the condition number of the full-order system matrix depends
quadratically on the gridwidth $h$. Even with powerful general-purpose preconditioners such as an
algebraic multigrid (AMG), the solution of the FOM thus remains a challenge.
However, once a reduced basis $\{w_i\}_{i\in\N}$ is obtained, the reduced system matrix
$\underline{Y}_\mu^{red}$ does no longer suffer from this problem.
As the basis is by construction orthonormal \wrt to the inner product $(\cdot,\cdot)_{\Honeb}$
and the norm equivalence $\norm{\cdot}_{\Honeb} \sim \norm{\cdot}_{\mu}$ holds for all $\mu$,
it is reasonable to assume that the condition of $\underline{Y}_\mu^{red}$
(which is precisely the Gram-matrix \wrt $(\cdot,\cdot)_\mu$) depends strongly on
the equivalence constants and less on the gridwidth.
This was confirmed by our numerical experiments where we did not observe an increase in condition
after grid refinement (while keeping the basis size constant) with values well below $10^2$.
\end{remark}

\subsection{Basis generation}\label{ssec:basisGeneration}
We base the computation of a reduced space $Y^N$ on a set $\{w^\delta_{\mu_i}\} \subset \ycal^\delta$ of $n_{train}$ solutions to the high-dimensional normal equation~\eqref{eq:discreteNormalEq}.
Subsequently, we want to employ the classic (weak) Greedy algorithm to find the dominant modes in the given set. As we are later interested in a good approximation of the reconstructed reduced solution, it would be natural to perform the algorithm on the set of reconstructions
$\{R_{\xcal}^{-1}A_{\mu_i}^*[w_{\mu_i}^\delta]\}$.
Interestingly, the standard greedy algorithm on the reconstructions can be written in terms of the test space snapshots as the norm $\norm{u_{\mu}^\delta - u_\mu^N}_{\xcal} = \norm{w_\mu^\delta - w_\mu^N}_\mu$ is computable (similar to~\cite[Algorithm~4.1]{BrunkenSmetanaUrban}). However, the orthonormalization of the generated basis then requires the evaluation of operators $A^{-*}_{\mu_1}A^*_{\mu_2}: \ycal^\delta \rightarrow\ycal^\delta$
which is only possible with a discretization of $\xcal$ which we want to avoid
\footnote{These operators appear when computing representations $w_\mu^\delta$ of differences $A_{\mu_i}^*[w^\delta_{\mu_i}] - \alpha A_{\mu_j}^*[w^\delta_{\mu_j}]$ \ie computing $w^\delta_\mu = A^{-*}_{\mu}(A_{\mu_i}^*[w^\delta_{\mu_i}] - \alpha A_{\mu_j}^*[w^\delta_{\mu_j}])$ for some $\mu$.}.
Thus, we propose to directly perform a greedy algorithm on the snapshots $\{w^\delta_{\mu_i}\}$
to construct a reduced basis of $\ycal^\delta$ (Alg.~\ref{alg:greedy}).
Since the snapshots originate from spaces equipped with different, parameter-dependent norms,
there is no canonical choice for the norm $\norm{\cdot}_*$ used in Alg.~\ref{alg:greedy}.
In our numerical experiments in Section~\ref{sec:numerics} we will exclusively consider the norm
$\norm{\cdot}_{\Honeb}$ on the parameter-independent space $\Honeb$.
However, depending on the problem at hand other choices such as the norm $\norm{\cdot}_{\mu^*}$
for a suitable fixed reference parameter $\mu^*$ might be viable as well.

\begin{algorithm}
\begin{algorithmic}
  \Require{training set $S_\mathrm{train} \subseteq\Pcal$, tolerance $\varepsilon$}
  \ForAll{$\mu\in S_\mathrm{train}$}
    \State Compute $w_\mu^\delta$ using~\eqref{eq:discreteNormalEq}
  \EndFor
  \Repeat
    \ForAll{$\mu\in S_\mathrm{train}$}
      \State Compute $w_\mu^N$ using~\eqref{eq:reducedNormalEq}
    \EndFor
    \State $\mu^* \leftarrow \text{argmax}_{\mu\in S_\mathrm{train}} \norm{w_\mu^\delta - w_\mu^N}_*$
    \State $S^{N+1} \leftarrow S^N \cup \{\mu^*\}$
    \State $Y^{N+1} \leftarrow \text{orthonormalize}(\{w_\mu^\delta,\; \mu\in S^{N+1}\})$
    \State $N \leftarrow N+1$
  \Until{$\max_{\mu\in S_\mathrm{train}}\norm{w_\mu^\delta - w_\mu^N}_* \leq \varepsilon$}
  \State \Return $Y^N$
\end{algorithmic}
\caption{Test space greedy algorithm}
\label{alg:greedy}
\end{algorithm}

\subsection{Error estimator}
In Algorithm~\ref{alg:greedy} the full solution $w_\mu^\delta$ needs to be computed for every parameter in the training set in order to determine the approximation error in the greedy loop. Typically, one wants to replace the true error by an effective and reliable error estimator. Choosing the $\Honeb$-norm on the test space we are in the well-understood coercive setting (see \eg\cite{haasdonkMORtutorial}) and an effective and reliable error estimator is given by the dual-norm of the residual, \ie
\begin{equation}\label{eq:errorEstimator}
  \Delta_N(\mu) \;:=\; \frac{\norm{r_\mu}_{\Honeb'}}{\sqrt{\alpha_{LB}(\mu)}},
\end{equation}
where the residual is defined as $r_\mu := f_\mu - a_\mu(w_\mu^N, \cdot) \in (\ycal^\delta)'$ and $\alpha_{LB}(\mu)$ denotes a lower bound on the coercivity constant of $\hat{a}_\mu(\cdot,\cdot)$. In our case, the true coercivity constant $\alpha(\mu)$ can be bounded from below by
\begin{equation*}
  \alpha(\mu) = \inf_{v\in\Honeb}\frac{\hat{a}_\mu(v,v)}{\norm{v}_{\Honeb}^2}
  = \inf_{v\in\Honeb}\frac{\norm{v}_\mu^2}{\norm{v}_{\Honeb}^2}
   \geq c(\mu)^2,
\end{equation*}
\ie the (squared) equivalence constant from Lemma~\ref{lemma:normEquivalence} which, as a reminder, is given by
\begin{equation*}
    c(\mu) = (2 + C_p^2(2\norm{c_\mu}_{\Linfty}^2+1))^{-\tfrac{1}{2}}.
\end{equation*}

\section{Numerical experiments}\label{sec:numerics}

In this section we present numerical experiments supporting the theoretical findings.
We first concentrate on the nonparametric problem and investigate the quality of the chosen
discretization scheme. Afterwards, we consider the parametrized problem and give results on the
proposed model order reduction approaches.\\
All of the high-fidelity computations were implemented in DUNE~\cite{bastian2021dune}
using the DUNE-PDELab discretization toolbox~\cite{bastian2010pdelab}.
For the implementation and evaluation of the parametrized model the
pyMOR-library~\cite{MilkRaveSchindler} was used.
The code for all numerical experiments together with instructions on how to reproduce the results
is publicly available at~\cite{reneltReactionAdvectionCode}.

We investigate the transport of a pollutant inside a catalytic filter.
To that end let the computational domain be defined as $\Omega := [0,1]^2$.
The chemical enters the filter at an inflow boundary $\Gin \subseteq\partial\Omega$ and needs to
pass the so-called washcoat $\Omega_w$ containing the catalyst. The reactive process therefore
takes place in this subregion. The transporting medium and the remaining pollutant then exit the
domain at $\Gout\subseteq\partial\Omega$. We assume that the washcoat material is coated by
an additional protecting layer $\Omega_c$. For the transporting velocity field we consider
two different flux models:

\subsubsection*{Flux model 1: Poiseuille flow}
In a first testcase we assume the flow to be laminar along the negative $y$-direction.
The velocity profile is then given by the explicit solution of the Hagen-Poiseuille-equation, \ie
\begin{equation}\label{eq:poiseuilleProfile}
  \vec{b}_0(r) \sim \frac{R^2 - r^2}{4\eta}, \quad \eta > 0,
\end{equation}
only depending on the distance to the central axis $\{x=\tfrac{1}{2}\}$,
\ie~$r(x,y) = |x-\tfrac{1}{2}|$.
We obtain the globally defined velocity field as
\begin{equation}\label{eq:poiseuilleFlow}
  \vec{b}(x,y) := (0, -\vec{b}_0(r(x,y)))^T.
  \tag{F1}\addtocounter{equation}{1}
\end{equation}
See Table~\ref{tab:paramsPoiseuille} and Fig.~\ref{fig:poiseuilleFigures} for a summary of the chosen parameters.
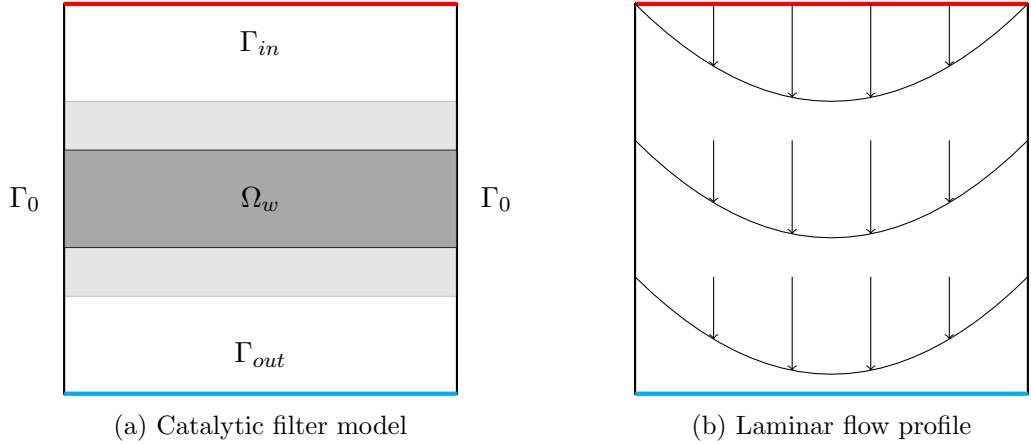
\begin{figure}[htb]
  \centering
  \begin{subfigure}{0.4\textwidth}
    \centering
    \begin{tikzpicture}[scale=0.7\textwidth/1cm]
      \draw[black,thick] (0.0,0.0) rectangle (1.0,1.0);
      \draw[red,ultra thick] (0.0, 1.0) -- (1.0, 1.0);
      \node[black] at (0.5, 0.9) {$\Gin$};
      \draw[cyan,ultra thick] (0.0, 0.0) -- (1.0, 0.0);
      \node[black] at (0.5, 0.1) {$\Gout$};
      \node[black] at (-0.1, 0.5) {$\Gamma_0$};
      \node[black] at (1.1, 0.5) {$\Gamma_0$};
      \filldraw[fill=black!40!white, draw=black, opacity=0.25] (0.0, 1./4) rectangle (1.0, 3./4);
      \filldraw[fill=black!40!white, draw=black, opacity=0.8] (0.0, 3./8) rectangle (1.0, 5./8);
      \node[black] at (0.5,0.5) {$\Omega_w$};
    \end{tikzpicture}
    \caption{Catalytic filter model}
    \label{fig:filterModelPoiseuille}
  \end{subfigure}
  \begin{subfigure}{0.4\textwidth}
    \centering
    \begin{tikzpicture}[scale=0.7\textwidth/1cm]
      \draw[black,thick] (0.0,0.0) rectangle (1.0,1.0);
      \draw[red,ultra thick] (0.0, 1.0) -- (1.0, 1.0);
      \draw[cyan,ultra thick] (0.0, 0.0) -- (1.0, 0.0);
      \foreach \y in {1, 0.65, 0.3} {
        \draw[color=black, domain=0:1] plot (\x,  {\y - (0.25 - (\x-0.5)^2)/(4*0.25)});
        % draw the arrows
        \foreach \i in {1,2,3,4} {
          \draw[->] ({\i / 5}, \y) -- ({\i / 5}, {\y - (0.25 - (\i/5 - 0.5)^2)/(4*0.25)});
        }
      }
    \end{tikzpicture}
    \caption{Laminar flow profile}
    \label{fig:laminarProfile}
  \end{subfigure}
  \caption{\label{fig:poiseuilleFigures} Geometric setup and flow profile for the Poiseuille model~\eqref{eq:poiseuilleFlow}}
\end{figure}
\begin{table}[htb]
  \centering
  \renewcommand{\arraystretch}{1.2}
  \begin{tabular}{|c|c|c|c|c|c|c|}
    \hline
    $\Omega_w$ & $\Omega_c$ & $\Gin$ & $\Gout$ & $R$ & $\eta$\\
    \hline
    $[0,1]\times [\tfrac{3}{8},\tfrac{5}{8}]$ & $([0,1]\times [\tfrac{1}{4},\tfrac{3}{4}]) \setminus \Omega_w$ & $[0,1]\times\{1\}$ & $[0,1]\times\{0\}$ & $0.5$ & $0.2$ \\
    \hline
  \end{tabular}
  \caption{Parameters for the Poiseuille-flow problem~\eqref{eq:poiseuilleFlow}}
  \label{tab:paramsPoiseuille}
\end{table}
\subsubsection*{Flux model 2: Darcy flow}
In a slightly more realistic approach we allow for more general in-/outflow boundaries.
In this case we may treat the catalytic filter as a porous medium, also allowing for different
premeabilities in the different compartments.
Viscious flow can then be modeled by Darcys law which states that the flux is proportional to the
gradient of the pressure $p$ of the fluid. Together with the continuity equation
(\ie $\vec{b}$ being divergence-free) and a no-flux-condition on the characteristic boundary
one obtains the full Darcy-model
\begin{equation}\label{eq:darcyFlow}
  \begin{cases}
    \phantom{\nabla\cdot{}} \vec{b} = -k\nabla p &\text{in}\;\Omega \\
    \nabla\cdot\vec{b} = 0 &\text{in}\;\Omega \\
    \phantom{\nabla\cdot{}} p = 1 &\text{on}\; \Gin, \\
    \phantom{\nabla\cdot{}} p = 0 &\text{on}\; \Gout, \\
    \phantom{\nabla\cdot{}} \vec{b} = 0 &\text{on}\; \Gamma_0 \;:= \Gamma \setminus (\Gin \cup \Gout).
  \end{cases}
  \tag{F2}\addtocounter{equation}{1}
\end{equation}
Here, $k \in \Linfty$ denotes the permeability coefficient which differs in the compartments
and is thus taken as a sum of indicator functions
\begin{equation*}
  k := \mathbbm{1}_{\Omega_{reac}} + k_w\mathbbm{1}_{\Omega_w} + k_c\mathbbm{1}_{\Omega_c}.
\end{equation*}
where the washcoat and its coating are less permeable \ie $k_c \leq k_w < 1$.\\
All chosen parameters are summarized in Table~\ref{tab:paramsDarcy}, also see Fig.~\ref{fig:darcyFigures}.

\begin{table}[htb]
  \centering
  \renewcommand{\arraystretch}{1.2}
  \begin{tabular}{|c|c|c|c|c|c|}
    \hline
    $\Omega_w$ & $\Omega_c$ & $\Gin$ & $\Gout$ & $k_w$ & $k_c$\\
    \hline
    $[0,1]\times [\tfrac{3}{8},\tfrac{5}{8}]$ & $([0,1]\times [\tfrac{1}{4},\tfrac{3}{4}]) \setminus \Omega_w$ & $\{0\}\times (\tfrac{3}{4},1)$ & $\{1\} \times (0,\tfrac{1}{4})$ & $0.2$ & $0.05$ \\
    \hline
  \end{tabular}
  \caption{\label{tab:paramsDarcy} Parameters for the Darcy-flow problem~\eqref{eq:darcyFlow}}
\end{table}

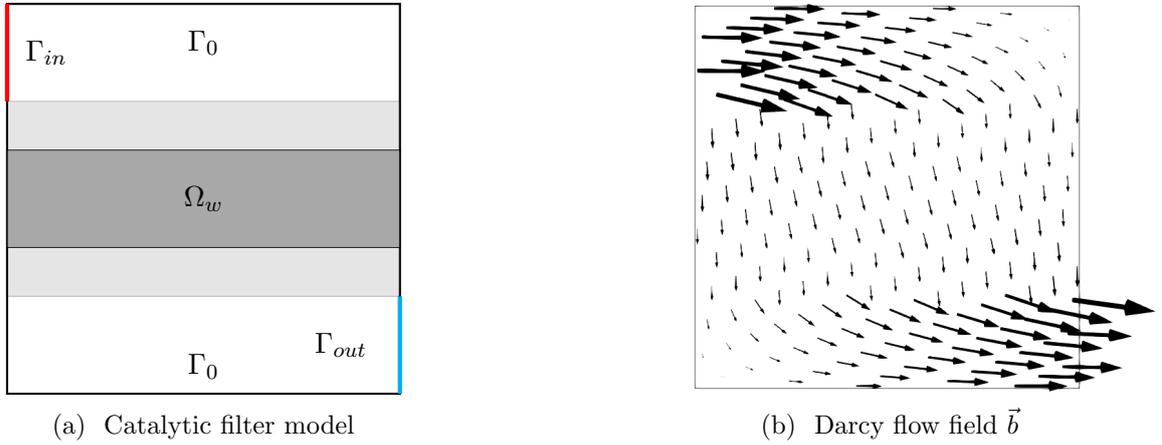
\begin{figure}[htb]
  \centering
  \begin{subfigure}{0.4\textwidth}
    \centering
    \begin{tikzpicture}[scale=0.7\textwidth/1cm]
      \draw[black,thick] (0.0,0.0) rectangle (1.0,1.0);
      \draw[red,ultra thick] (0.0, 1.0) -- (0.0, 3./4);
      \node[black] at (0.1, 7/8) {$\Gin$};
      \draw[cyan,ultra thick] (1.0, 0.0) -- (1.0, 1./4);
      \node[black] at (0.85, 1/8) {$\Gout$};
      \node[black] at (0.5, 0.07) {$\Gamma_0$};
      \node[black] at (0.
      5, 0.9) {$\Gamma_0$};
      \filldraw[fill=black!40!white, draw=black, opacity=0.25] (0.0, 1./4) rectangle (1.0, 3./4);
      \filldraw[fill=black!40!white, draw=black, opacity=0.8] (0.0, 3./8) rectangle (1.0, 5./8);
      \node[black] at (0.5,0.5) {$\Omega_w$};
    \end{tikzpicture}
    \caption{\label{fig:filterModel} Catalytic filter model}
  \end{subfigure}
  \hspace{0.1\linewidth}
  \begin{subfigure}{0.4\linewidth}
    \centering
    \includegraphics[height=0.7\linewidth]{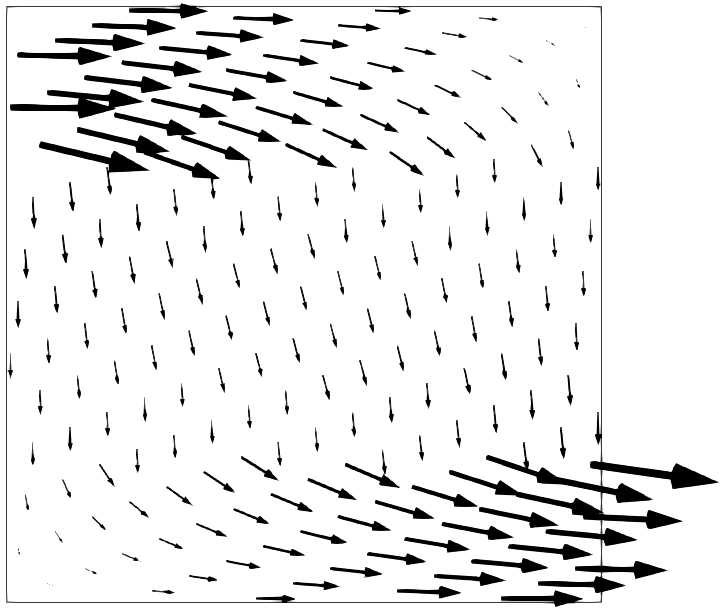}
    \caption{\label{fig:darcyStreamlines} {Darcy flow field $\vec{b}$ \hspace*{8mm}} }
  \end{subfigure}
  \caption{\label{fig:darcyFigures} Geometric setup and flow field for the Darcy model~\eqref{eq:darcyFlow}}
\end{figure}

\newpage
\subsubsection*{Specification of the non-parametric advection-reaction problem}
Having now determined the flow field we specify the remaining data functions for
Problem~\eqref{eq:transport:explicitContinuousNormalEq}.
For the boundary values we use a parametrization $\Gin = \phi([0,1])$ by an isomorphism $\phi$ and set
$g_D := \phi \circ \hat{g}_D$ for some inflow concentration $\hat{g}_D: [0,1]\rightarrow \R$.
Similar to the permeability, the reaction coefficient is assumed to be piecewise constant in the
compartments $\Omega_w$ and $\Omega_c$ and zero elsewhere and thus defined as $c := c_w \cdot\mathbbm{1}_{\Omega_w} + c_c \cdot\mathbbm{1}_{\Omega_c}$.
In this example we do not consider additional sources or sinks.
All parameter values can be found in Table~\ref{tab:params}.
\begin{table}[htb]
  \centering
  \renewcommand{\arraystretch}{1.2}
  \renewcommand{\therowcntr}{T.\arabic{rowcntr}}
  \begin{tabular}{|N|c|c|c|c|c|}
    \hline
     \multicolumn{1}{|c|}{Testcase} & $\vec{b}$ & $\hat{g}_D(s)$ & $f_\circ(x)$ & $c_w$ & $c_c$ \\
    \hline
    \label{params:testcase1} & \eqref{eq:poiseuilleFlow}
    & $\sin(4\pi s)^2$ & \multirow{3}{*}{$\equiv 0$} &
    \multirow{3}{*}{$\equiv 0.5$}  & \multirow{3}{*}{$\equiv 0.1$} \\
    \label{params:testcase2} & \eqref{eq:poiseuilleFlow} &
    $\mathbbm{1}_{[0.25,0.75]}(s)$ & & & \\
    \label{params:testcase3} & \eqref{eq:darcyFlow} &
    $\sin(\pi s)^2$ & & & \\
    \hline
  \end{tabular}
  \caption{Chosen parameters for the advection-reaction problem}
  \label{tab:params}
\end{table}

\subsubsection*{Discretization}
For simplicity reasons we use a structured quadrilateral mesh $\mathcal{T}_h$ with fixed
gridwidth $h$ (although this is not required for the method).
We then discretize the optimal test space $\ycal$ with the Lagrange finite element
space $\mathbb{Q}^k(\mathcal{T}_h)$ consisting of the globally continuous and piecewise polynomial
(of at most order $k$ in each variable) functions.
Since $\mathbb{Q}^k(\mathcal{T}_h)\hookrightarrow H^1(\Omega)\hookrightarrow\ycal$,
this is a conforming discretization,
even in the special case $T_{min}=0$.
We compute solutions for gridwidths $h_r := 2^{-(r+3)},\, r=0,..,r_{max}$ such that
$h_0\!=\!\tfrac{1}{8}$ resolves the geometry.
For the error computation a solution obtained by an $(k\!+\!1)$-th order DG-scheme with upwind flux computed
on the refinement level $r_{max}\!+\!1$ is used.
Similarly, the Darcy-velocity field in Testcase~\ref{params:testcase3} is obtained by an SIPG scheme on the finest
refinement level in order to avoid the introduction of any additional error contributions.
\par
For the analytically given Poiseuille-field and $C^\infty$-boundary data we observe an
quasi-optimal order of almost $k+1$ for both bilinear and biquadratic finite elements.
In the case of discontinuous inflow data the error still convergences, although the rate drops
significantly which is to be expected (Fig.~\ref{fig:hconvergence:poiseuille}).
When switching to the flow field obtained by solving
the Darcy-equation, we get suboptimal rates of approximately $k$ (Fig.~\ref{fig:hconvergence:darcy}),
which coincides with the
lower bound on the rate as discussed in Remark~\ref{rem:transport:convergenceRates}.

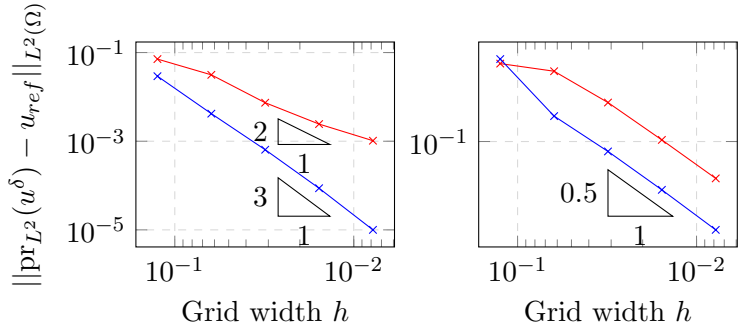
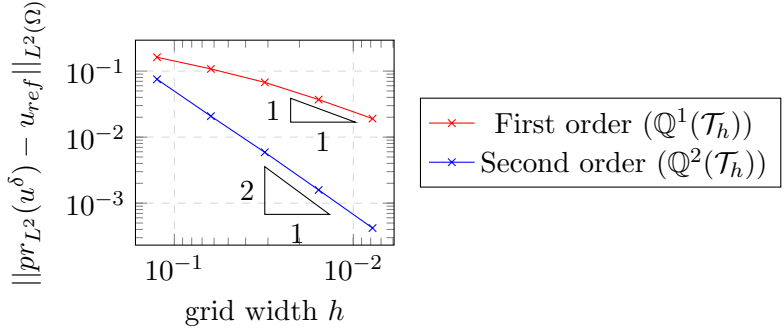
\begin{figure}[ht]
  \centering
  \begin{subfigure}{0.6\linewidth}
    \begin{tikzpicture}
      \begin{loglogaxis}[
          name=ax1,
          width=0.45\linewidth, % Scale the plot to \linewidth
          grid=major, % Display a grid
          grid style={dashed,gray!30}, % Set the style
          xlabel=Grid width $h$,
          x dir=reverse,
          ylabel= $\norm{\mathrm{pr}_{L^2}(u^\delta) - u_{ref}}_{\Ltwo}$,
        ]
        \addplot[color=red, mark=x, forget plot]
        table[x=gridwidth,y=l2error,col sep=comma]
        {data/convergenceTests/h-convergence_fo_poiseuille_ref4.csv};
        \logLogSlopeTriangle{0.55}{0.2}{0.5}{2}{black}
        \addplot[color=blue, mark=x, forget plot]
        table[x=gridwidth,y=l2error,col sep=comma]
        {data/convergenceTests/h-convergence_so_poiseuille_ref4.csv};
        \logLogSlopeTriangle{0.55}{0.2}{0.15}{3}{black}
      \end{loglogaxis}
      \begin{loglogaxis}[
          at={(ax1.south east)},
          xshift = 0.1\linewidth,
          width=0.45\linewidth, % Scale the plot to \linewidth
          grid=major, % Display a grid
          grid style={dashed,gray!30}, % Set the style
          xlabel=Grid width $h$,
          x dir=reverse,
        ]
        \addplot[color=red, mark=x]
        table[x=gridwidth,y=l2error,col sep=comma]
        {data/convergenceTests/h-convergence_fo_poiseuille_ref4_l2data.csv};
        \addplot[color=blue, mark=x]
        table[x=gridwidth,y=l2error,col sep=comma]
        {data/convergenceTests/h-convergence_so_poiseuille_ref4_l2data.csv};
        \logLogSlopeTriangle{0.5}{0.25}{0.15}{0.5}{black}
      \end{loglogaxis}
    \end{tikzpicture}
    \caption{\label{fig:hconvergence:poiseuille} Convergence rates for the laminar flux model. For smooth inflow data (left, Testcase~\ref{params:testcase1}) one obtains almost optimal convergence orders. Discontinuous inflow data (right, Testcase~\ref{params:testcase2}) limits the rate but still converges.}
  \end{subfigure}
  \begin{subfigure}{0.6\linewidth}
  \centering
    \begin{tikzpicture}
      \begin{loglogaxis}[
          width=0.45\linewidth, % Scale the plot to \linewidth
          grid=major, % Display a grid
          grid style={dashed,gray!30}, % Set the style
          xlabel=Global basis size, % Set the labels
          xlabel=grid width $h$,
          x dir=reverse,
          ylabel= $\norm{pr_{L^2}(u^\delta) - u_{ref}}_{\Ltwo}$,
          legend style={at={(1.1,0.5)},anchor=west},
        ]
        \addplot[color=red, mark=x]
        table[x=gridwidth,y=l2error,col sep=comma]
        {data/convergenceTests/h-convergence_fo_darcy_ref4.csv};
        \addlegendentry{First order ($\mathbb{Q}^1(\mathcal{T}_h)$)}
        \logLogSlopeTriangle{0.6}{0.25}{0.6}{1}{black}
        \addplot[color=blue, mark=x]
        table[x=gridwidth,y=l2error,col sep=comma]
        {data/convergenceTests/h-convergence_so_darcy_ref4.csv};
        \addlegendentry{Second order ($\mathbb{Q}^2(\mathcal{T}_h)$)}
        \logLogSlopeTriangle{0.5}{0.25}{0.15}{2}{black}
      \end{loglogaxis}
    \end{tikzpicture}
    \caption{\label{fig:hconvergence:darcy} Convergence rates for reactive transport goverened by
    Darcy-flux (Testcase~\ref{params:testcase3}).
    For the computation of the advection field a high-order scheme on a fine mesh was used in order
    to prevent additional error contributions.}
  \end{subfigure}
  \caption{Convergence of the $L^2(\Omega)$-error under $h$-refinement.
  Top: Poiseuille model~\eqref{eq:poiseuilleFlow}, Bottom: Darcy model~\eqref{eq:darcyFlow}}
\end{figure}

\FloatBarrier

\subsection{The parametrized problem}\label{ssec:numerics:parametrized}
We now exclusively consider Testcase~\ref{params:testcase3} parametrized by the
reaction coefficients $c_w$ and $c_c$ and the magnitude $g_0$ of the inflow profile, i.e.\
\begin{equation*}
  \hat{g}_{D,\mu}: [0,1] \rightarrow \R, \qquad \hat{g}_{D,\mu}(s) := g_0\sin(\pi s)^2.
\end{equation*}
The parameter $\mu$ therefore consists of the three components $\mu = (c_w, c_c, g_0)$ for which we
specify three different (compact) parameter domains $\Pcal\subset\R^3$ (Tab.~\ref{tab:paramsParametrizedProblem}).
In a first testcase we consider a two-compartment model without reaction in the coating layer and fixed $g_0$. In Testcase~\ref{params:pTestcase2} we allow $c_c$ to vary but still enforce the
constraint $c_c \leq c_w$.
By additionally varying the inflow magnitude $g_0$ we obtain our last example~\ref{params:pTestcase3}.
In Fig.~\ref{fig:sampleSolutions} solutions for different parameter combinations are shown.

\begin{figure}[ht]
  \centering
  \begin{subfigure}{0.2\linewidth}
    \centering
    \includegraphics[width=0.9\linewidth]{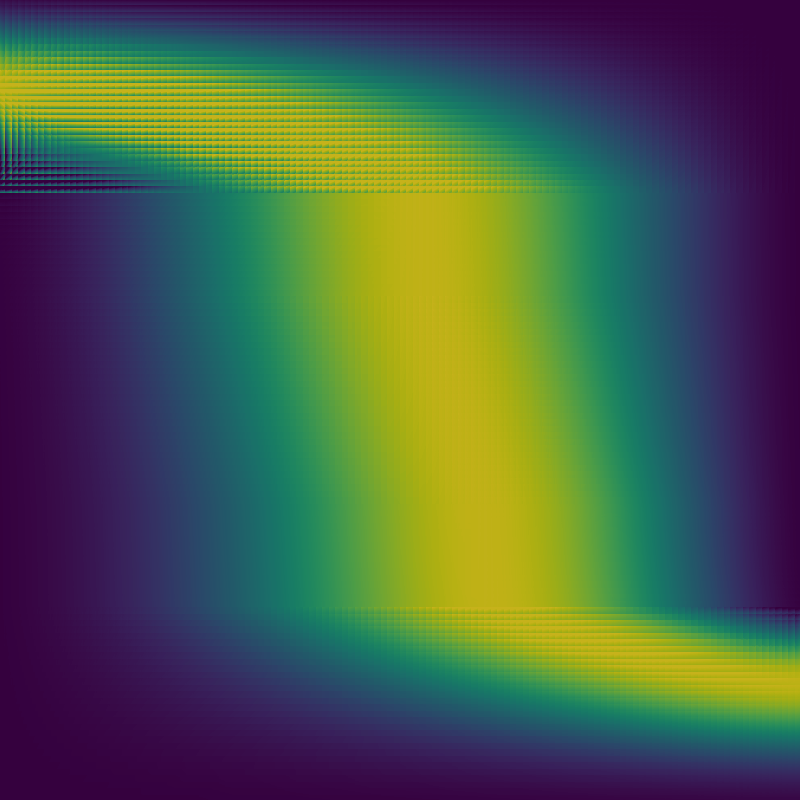}
    \caption{$\mu=(0,0,1)$}
  \end{subfigure}
  \begin{subfigure}{0.2\linewidth}
    \centering
    \includegraphics[width=0.9\linewidth]{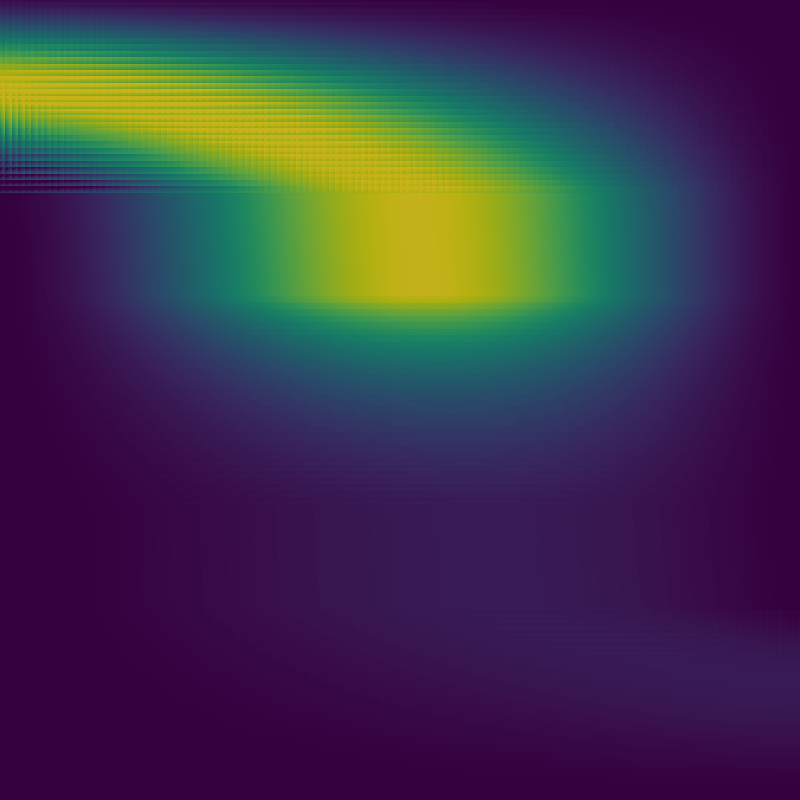}
    \caption{$\mu=(1,0,1)$}
  \end{subfigure}
  \begin{subfigure}{0.2\linewidth}
    \centering
    \includegraphics[width=0.9\linewidth]{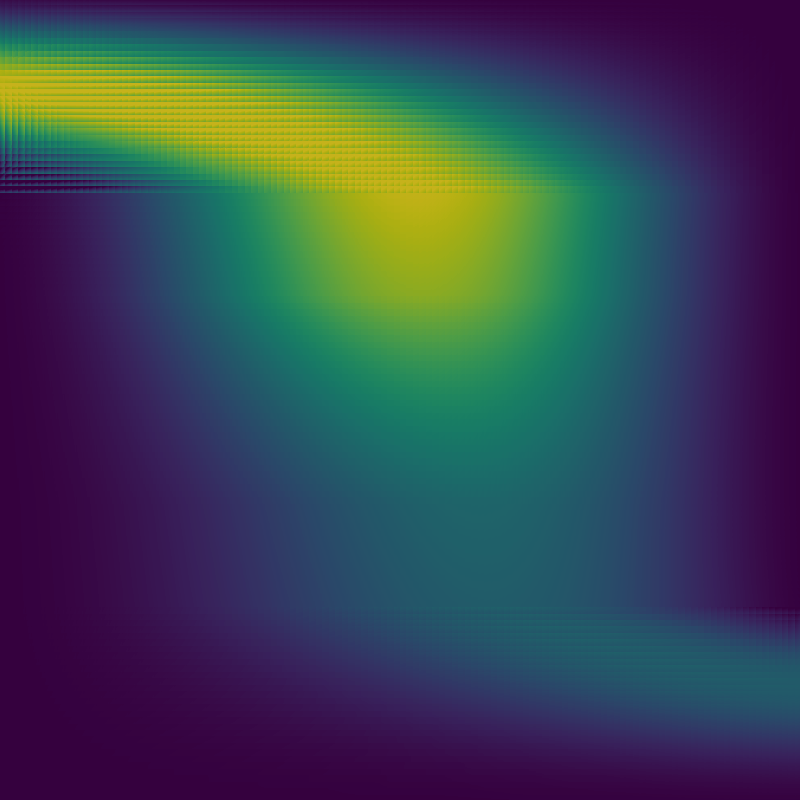}
  \caption{$\mu=(0.3,0.1,1)$}
  \end{subfigure}
  \hspace{0.02\linewidth}
  \begin{subfigure}{0.1\linewidth}
  \begin{tikzpicture}
    \pgfplotscolorbardrawstandalone[
      colormap/viridis,
      point meta min = 0.0,
      point meta max = 1.0,
      colorbar style = {height=0.2*0.9/0.1*\linewidth, width=0.3\linewidth}
      ]
  \end{tikzpicture}
  \\\phantom{(d)}
  \end{subfigure}
  \caption{\label{fig:sampleSolutions} Solutions $u_\mu$ to the parametrized problem~\eqref{eq:parametrizedStrongProblem} for different values of the parameter $\mu =(c_w,c_c,g_0)$. }
\end{figure}

\begin{table}[htb]
  \centering
  \renewcommand{\arraystretch}{1.2}
  \renewcommand{\therowcntr}{P.\arabic{rowcntr}}
  \begin{tabular}{|N|c|c|c|c|c|}
    \hline
    \multicolumn{1}{|c|}{Testcase} & Parameter domain $\Pcal$ & $n_{train}$ \\
    \hline
    \label{params:pTestcase1} & $[0.0, 1.0] \times \{0\} \times \{1\}$ & $500$ \\
    \label{params:pTestcase2} & $\{ 0.0 \,\leq\, c_c \,\leq\, c_w \,\leq\, 1.0\} \times \{1\}$ & $630$\\
    \label{params:pTestcase3} & $\{ 0.0 \,\leq\, c_c \,\leq\, c_w \,\leq\, 1.0\} \times [1,10]$  & $6300$ \\
    \hline
  \end{tabular}
  \caption{Specification of the parameter domains}
  \label{tab:paramsParametrizedProblem}
\end{table}

\subsubsection*{Investigation of the approximation error decay}
For all testcases we generate reduced spaces using a weak greedy algorithm with the error estimator~\eqref{eq:errorEstimator}. We investigate both the approximation error (Fig.~\ref{fig:validationError}) - which we define as the
$L^2(\Omega)$-error $\Lnorm{\operatorname{pr}_{L^2}(u_\mu^\delta- u_\mu^N)}$ between the FOM-solution
$u_\mu^\delta$ and the ROM-solution $u_\mu^N$ - as well as the runtime of the reduced model
(Fig.~\ref{fig:runtimeEvaluation}).
\par
In all three testcases we see the expected exponential decay of the approximation error. For the first
testcase with one varying parameter we observe a rate of $\mathcal{O}(\exp(-\beta N))$
with an exponent $\beta\approx 1.75$ while in the second testcase we have $\beta\approx 0.7$.
If we additionally include the
magnitude of the inflow profile in the parametrization, the rate is not affected. This is to be expected
as this parameter only occurs in the right-hand side and only requires one additional reduced basis
function. However, the absolute error increases by approximately one magnitude, simply due to the fact
that by rescaling the solution we also scale the error by the same factor.
\par
Compared to the full order model (which takes approximately $0.5s$ to solve) the reduced model
is approximately $10^3$ times faster while achieving low approximation errors even for very small
reduced basis sizes.
\begin{figure}[ht]
  \centering
  \begin{subfigure}{\linewidth}
    \begin{tikzpicture}
      \begin{semilogyaxis}[
          boxplot/draw direction = y,
          boxplot = {
                  draw position = {0.25 + round(\plotnumofactualtype/3 - 0.49)
                                        + 1/4*mod(\plotnumofactualtype, 3)},
                  box extend = 1/4},
          width=0.9\linewidth,
          height=0.25\linewidth,
          xmajorgrids=true,
          grid style={gray!30, dashed},
          axis x line = bottom,
          axis y line = left,
          xmin=0.0, xmax=15, ymin=1e-14,ymax=1e1,
          xlabel= Reduced basis size $N$,
          xtick = {0.5,1.5,...,13.5},
          xticklabels = {1,...,14},
          %x tick label as interval,
          ytick={1e-1,1e-4,1e-7,1e-10,1e-13},
          restrict y to domain=1e-10:1e-1,
          ylabel= {Approximation error},
          legend style={at={(0.95,0.6)},anchor=west},
          area legend,
          legend entries={\ref{params:pTestcase1},
          \ref{params:pTestcase2},
          \ref{params:pTestcase3}}
        ]
        \foreach \n in {1,...,14} {
          \addplot[color=black,fill=Goldenrod,boxplot, /pgfplots/boxplot/hide outliers]
          table[y={error_dim_\n}, fill,col sep=comma]
          {data/weak_greedy/rb_evaluation_ntrain_500_P1_H1b_ntest_500.csv};
          \addplot[color=black,fill=Magenta,boxplot, /pgfplots/boxplot/hide outliers]
          table[y={error_dim_\n}, fill,col sep=comma]
          {data/weak_greedy/rb_evaluation_ntrain_35_P2_H1b_ntest_500.csv};
          \addplot[color=black,fill=Cyan,boxplot, /pgfplots/boxplot/hide outliers]
          table[y={error_dim_\n}, fill,col sep=comma]
          {data/weak_greedy/rb_evaluation_ntrain_35_P3_H1b_ntest_500.csv};
        }
      \end{semilogyaxis}
    \end{tikzpicture}
    \caption{\label{fig:validationError} Approximation error of the reduced solution compared to a high-fidelity SIPG-solution computed on an additionally refined grid.}
    \vfill
  \end{subfigure}
  \begin{subfigure}{\linewidth}
    \begin{tikzpicture}
      \begin{axis}[
          boxplot/draw direction = y,
          boxplot = {
                  draw position = {0.25 + round(\plotnumofactualtype/3 - 0.49)
                                        + 1/4*mod(\plotnumofactualtype, 3)},
                  box extend = 1/4},
          width=0.9\linewidth,
          height=0.25\linewidth,
          xmajorgrids=true,
          grid style={gray!30, dashed},
          axis x line = bottom,
          axis y line = left,
          xmin=0.0, xmax=15, ymin=3.6e-4,ymax=5e-4,
          xlabel= Reduced basis size $N$,
          xtick = {0.5,1.5,...,13.5},
          xticklabels = {1,...,14},
          ytick={3.6e-4, 4.0e-4, 4.4e-4, 4.8e-4},
          restrict y to domain=1e-4:1e-3,
          ylabel= {Solving time [s]},
          legend style={at={(.95,0.6)},anchor=west},
          area legend,
          legend entries={\ref{params:pTestcase1},
          \ref{params:pTestcase2},
          \ref{params:pTestcase3}},
        ]
        \foreach \n in {1,...,14} {
          \addplot[color=black,fill=Goldenrod,boxplot, /pgfplots/boxplot/hide outliers]
          table[y={rom_time_dim_\n}, fill,col sep=comma]
          {data/weak_greedy/rb_evaluation_ntrain_500_P1_H1b_ntest_500.csv};
          \addplot[color=black,fill=Magenta,boxplot, /pgfplots/boxplot/hide outliers]
          table[y={rom_time_dim_\n}, fill,col sep=comma]
          {data/weak_greedy/rb_evaluation_ntrain_35_P2_H1b_ntest_500.csv};
          \addplot[color=black,fill=Cyan,boxplot, /pgfplots/boxplot/hide outliers]
          table[y={rom_time_dim_\n}, fill,col sep=comma]
          {data/weak_greedy/rb_evaluation_ntrain_35_P3_H1b_ntest_500.csv};
        }
      \end{axis}
    \end{tikzpicture}
    \caption{\label{fig:runtimeEvaluation}
    Runtime for solving the reduced system on an Intel~Xeon~Gold 6254 @$3.1$ Ghz.
    Solving the full model takes around $0.5$s (speedup: $\approx 10^3$).}
    \vfill
  \end{subfigure}
  \caption{Evaluation of the reduced models on a test set of $n_{test}=500$ additional
  randomly chosen parameters. The reduced model was obtained using the $\Honeb$-norm.}
\end{figure}
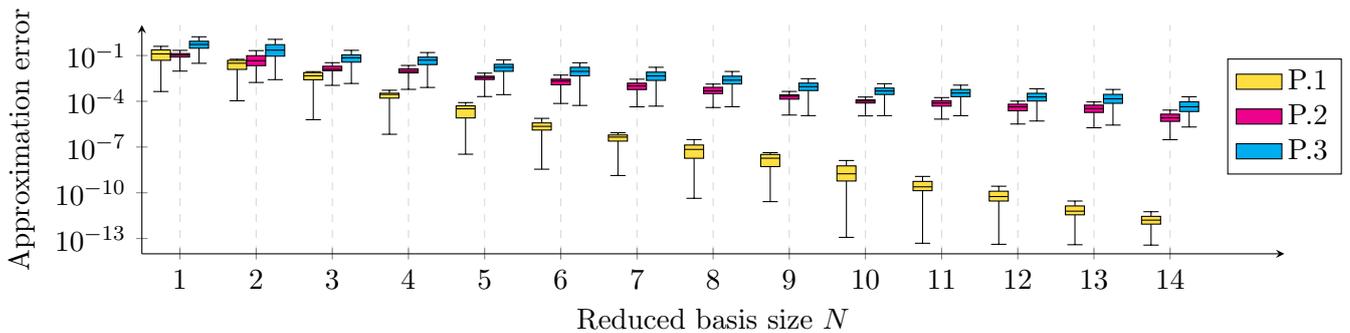
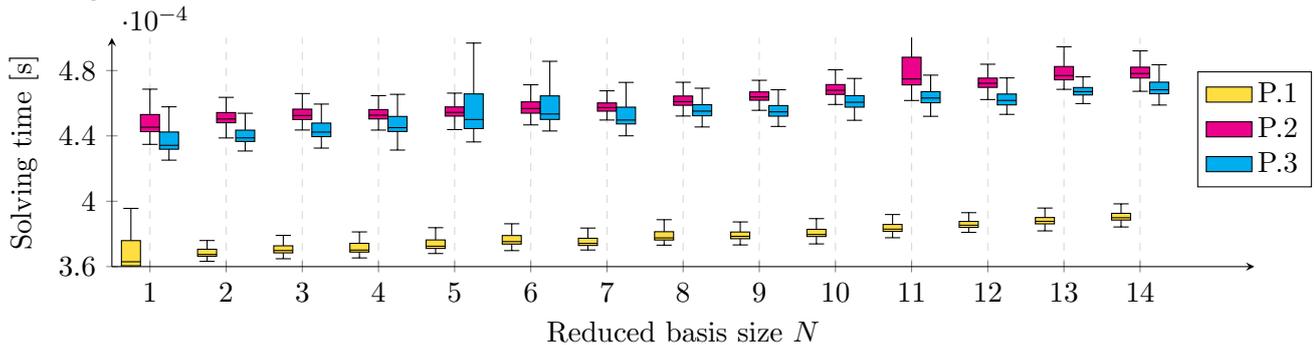

\FloatBarrier

\section{Conclusion}\label{sec:conclusion}
In this article we presented a model order reduction approach for reactive transport based on an ultraweak variational formulation. Choosing the optimal test space containing all supremizers and an operator-, and thus parameter-dependent norm leads to an optimally-stable scheme which can be reformulated as a normal equation on the space of test functions. We were able to prove that equivalence of the parameter-dependent norm on the test functions to a parameter-independent norm implies the exponential approximability of the solution set. This directly translates to the set of reconstructed primal solutions as well.

Moreover, we proposed to use a greedy algorithm to generate reduced approximation spaces for the test space and solve for a reduced solution of the adjoint normal equation. Similar to the full order scheme, the reconstruction of the primal reduced solution is then replaced by functional evaluations of the dual reduced solution avoiding an explicit interpolation of the non-standard trial space. Numerical experiments confirmed that the generated spaces indeed approximate the high-fidelity solutions with an error exponentially decaying with an increasing number of reduced basis functions.

While this contribution was focused on the advection-reaction equation, investigating the
approximability of parametrized non-selfadjoint problems by analyzing the parameter-dependent
optimal test space norm is a technique that can be directly applied to other problems. In particular,
Friedrichs'-systems (which include linear elasticity, Maxwells equations, linearized Navier-Stokes \etc) can all be formulated in an ultraweak form to which our work should thus naturally extend. Similarly, time-dependent problems can be tackled by switching to a space-time formulation. The numerical solving of the normal equation remains a challenge and will also require future work.

\section*{Acknowledgements}
The authors acknowledge funding by the BMBF under contract 05M20PMA
and by the Deutsche Forschungsgemeinschaft under Germany’s Excellence
Strategy EXC 2044 390685587, Mathematics M\"unster: Dynamics --
Geometry -- Structure.

\bibliography{ms}

\begin{thebibliography}{10}

\bibitem{arbes2023kolmogorov}
Florian Arbes, Constantin Greif, and Karsten Urban.
\newblock The {K}olmogorov {$N$}-width for linear transport: {E}xact
  representation and the influence of the data.
\newblock {\em arXiv preprint arXiv:2305.00066}, 2023.

\bibitem{babuska1971error}
Ivo Babu{\v{s}}ka.
\newblock Error-bounds for finite element method.
\newblock {\em Numer. Math.}, 16(4):322--333, 1971.

\bibitem{bastian2021dune}
Peter Bastian, Markus Blatt, Andreas Dedner, Nils-Arne Dreier, Christian
  Engwer, Ren\'{e} Fritz, Christoph Gr\"{u}ninger, Dominic Kempf, Robert
  Kl\"{o}fkorn, Mario Ohlberger, and Oliver Sander.
\newblock The {{DUNE}} framework: basic concepts and recent developments.
\newblock {\em Comput. Math. Appl.}, 81:75--112, 2021.

\bibitem{bastian2010pdelab}
Peter Bastian, Felix Heimann, and Sven Marnach.
\newblock Generic implementation of finite element methods in the {D}istributed
  and {U}nified {N}umerics {E}nvironment ({DUNE}).
\newblock {\em Kybernetika}, 46(2):294--315, 2010.

\bibitem{BennerOhlbergerCohenWillcox}
Peter Benner, Mario Ohlberger, Albert Cohen, and Karen Willcox.
\newblock {\em Model Reduction and Approximation}.
\newblock Society for Industrial and Applied Mathematics, Philadelphia, PA,
  2017.

\bibitem{broersen2018transportDPG}
Dirk Broersen, Wolfgang Dahmen, and Rob Stevenson.
\newblock On the stability of {DPG} formulations of transport equations.
\newblock {\em Math. Comp.}, 87(311):1051--1082, 2018.

\bibitem{brooks1982supg}
Alexander~N Brooks and Thomas~JR Hughes.
\newblock Streamline upwind/{P}etrov-{G}alerkin formulations for convection
  dominated flows with particular emphasis on the incompressible
  {N}avier-{S}tokes equations.
\newblock {\em Comput. Methods Appl. Mech. Engrg.}, 32(1-3):199--259, 1982.

\bibitem{BrunkenSmetanaUrban}
J.~Brunken, K.~Smetana, and K.~Urban.
\newblock ({P}arametrized) first order transport equations: realization of
  optimally stable {P}etrov-{G}alerkin methods.
\newblock {\em SIAM J. Sci. Comput.}, 41(1):A592--A621, 2019.

\bibitem{burela2023spod}
Shubhaditya Burela, Philipp Krah, and Julius Reiss.
\newblock Parametric model order reduction for a wildland fire model via the
  shifted {POD} based deep learning method.
\newblock {\em arXiv preprint arXiv:2304.14872}, 2023.

\bibitem{caiFOSLL}
Zhiqiang Cai, Thomas~A Manteuffel, Stephen~F McCormick, and John Ruge.
\newblock First-order system {$\mathcal{L}\mathcal{L}^*$} ({FOSLL*}): Scalar
  {E}lliptic {P}artial {D}ifferential {E}quations.
\newblock {\em SIAM J. Numer. Anal.}, 39(4):1418--1445, 2001.

\bibitem{ChanHeuer}
Jesse Chan, Norbert Heuer, Tan Bui-Thanh, and Leszek Demkowicz.
\newblock A robust {DPG} method for convection-dominated diffusion problems
  {II}: Adjoint boundary conditions and mesh-dependent test norms.
\newblock {\em Comput. Math. Appl.}, 67(4):771--795, 2014.

\bibitem{DahmenHuangSchwab}
W.~Dahmen, C.~Huang, C.~Schwab, and G.~Welpers.
\newblock Adaptive {P}etrov-{G}alerkin methods for {F}irst order transport
  equations.
\newblock {\em SIAM J. Num. Anal.}, 50(5):2420--2445, 2012.

\bibitem{DahmenPleskenWelper}
Wolfgang Dahmen, Christian Plesken, and Gerrit Welper.
\newblock Double greedy algorithms: {R}educed basis methods for transport
  dominated problems.
\newblock {\em ESAIM: Math. Model. Numer. Anal.}, 48(3):623--663, 2014.

\bibitem{demkowiczDPG2}
Leszek Demkowicz and Jay Gopalakrishnan.
\newblock A class of discontinuous {P}etrov--{G}alerkin methods. {II}. optimal
  test functions.
\newblock {\em Numerical Methods for Partial Differential Equations},
  27(1):70--105, 2011.

\bibitem{demkowiczDPGstar}
Leszek Demkowicz, Jay Gopalakrishnan, and Brendan Keith.
\newblock The {DPG}-star method.
\newblock {\em Comput. Math. Appl.}, 79(11):3092--3116, 2020.

\bibitem{demkowiczDPG1}
Leszek Demkowicz and Jayadeep Gopalakrishnan.
\newblock A class of discontinuous {P}etrov--{G}alerkin methods. {P}art {I}:
  The transport equation.
\newblock {\em Comput. Methods Appl. Mech. Engrg.}, 199(23-24):1558--1572,
  2010.

\bibitem{ferrero2022registration}
Andrea Ferrero, Tommaso Taddei, and Lei Zhang.
\newblock Registration-based model reduction of parameterized two-dimensional
  conservation laws.
\newblock {\em J. Comput. Phys.}, 457:111068, 2022.

\bibitem{friedrichs1958}
Kurt~Otto Friedrichs.
\newblock Symmetric positive linear differential equations.
\newblock {\em Commun. Pure Appl. Math.}, 11(3):333--418, 1958.

\bibitem{greif2019kolmogorov}
Constantin Greif and Karsten Urban.
\newblock Decay of the {K}olmogorov {N}-width for wave problems.
\newblock {\em Appl. Math. Lett.}, 96:216--222, 2019.

\bibitem{haasdonkMORtutorial}
Bernard Haasdonk.
\newblock {\em Chapter 2: Reduced Basis Methods for Parametrized PDEs—A
  Tutorial Introduction for Stationary and Instationary Problems}, pages
  65--136.
\newblock Society for Industrial and Applied Mathematics, Philadelphia, PA,
  2017.

\bibitem{hain2022ultraweak}
Stefan Hain and Karsten Urban.
\newblock An {U}ltra-{W}eak {S}pace-{T}ime {V}ariational {F}ormulation for the
  {S}chr\"odinger {E}quation.
\newblock {\em arXiv preprint arXiv:2212.14398}, 2022.

\bibitem{HenningPalitta}
Julian Henning, Davide Palitta, Valeria Simoncini, and Karsten Urban.
\newblock An ultraweak space-time variational formulation for the wave
  equation: {A}nalysis and efficient numerical solution.
\newblock {\em ESAIM: Math. Model. Numer. Anal.}, 56(4):1173--1198, 2022.

\bibitem{hesthaven2018nonintrusive}
Jan~S Hesthaven and Stefano Ubbiali.
\newblock Non-intrusive reduced order modeling of nonlinear problems using
  neural networks.
\newblock {\em J. Comput. Phys.}, 363:55--78, 2018.

\bibitem{hughes1979sdfem}
Thomas~JR Hughes.
\newblock A multidimentional upwind scheme with no crosswind diffusion.
\newblock {\em Finite element methods for convection dominated flows, AMD 34},
  1979.

\bibitem{hughes1989gls}
Thomas~J.R. Hughes, Leopoldo~P. Franca, and Gregory~M. Hulbert.
\newblock A new finite element formulation for computational fluid dynamics:
  {VIII}. the galerkin/least-squares method for advective-diffusive equations.
\newblock {\em Comput. Methods Appl. Mech. Engrg.}, 73(2):173--189, 1989.

\bibitem{keithAPriori}
Brendan Keith.
\newblock A priori error analysis of high-order {LL*}({FOSLL*}) finite element
  methods.
\newblock {\em Comput. Math. Appl.}, 103:12--18, 2021.

\bibitem{kim2020nonlinear}
Youngkyu Kim, Youngsoo Choi, David Widemann, and Tarek Zohdi.
\newblock Efficient nonlinear manifold reduced order model.
\newblock {\em arXiv preprint arXiv:2011.07727}, 2020.

\bibitem{LeeCarlberg2020}
Kookjin Lee and Kevin~T Carlberg.
\newblock Model reduction of dynamical systems on nonlinear manifolds using
  deep convolutional autoencoders.
\newblock {\em J. Comput, Phys.}, 404:108973, 2020.

\bibitem{MilkRaveSchindler}
Ren\'{e} Milk, Stephan Rave, and Felix Schindler.
\newblock py{MOR} -- {G}eneric {A}lgorithms and {I}nterfaces for {M}odel
  {O}rder {R}eduction.
\newblock {\em SIAM J. Sci. Comput.}, 38(5):S194--S216, 2016.

\bibitem{mirhoseini2023model}
Marzieh~Alireza Mirhoseini and Matthew~J Zahr.
\newblock Model reduction of convection-dominated partial differential
  equations via optimization-based implicit feature tracking.
\newblock {\em J. Comput. Phys.}, 473:111739, 2023.

\bibitem{OhlbergerRave}
Mario Ohlberger and Stephan Rave.
\newblock Reduced basis methods: Success, limitations and future challenges.
\newblock In {\em Proceedings of ALGORITMY}, pages 1--12, 2016.

\bibitem{philip1970flow}
JR~Philip.
\newblock Flow in porous media.
\newblock {\em Annual Review of Fluid Mechanics}, 2(1):177--204, 1970.

\bibitem{reiss2018spod}
Julius Reiss, Philipp Schulze, J{\"o}rn Sesterhenn, and Volker Mehrmann.
\newblock The shifted proper orthogonal decomposition: A mode decomposition for
  multiple transport phenomena.
\newblock {\em SIAM J. Sci. Comp.}, 40(3):A1322--A1344, 2018.

\bibitem{reneltReactionAdvectionCode}
Lukas Renelt.
\newblock Source code to '{M}odel order reduction of an ultraweak and optimally
  stable variational formulation for parametrized reactive transport problems',
  2023.

\bibitem{renelt2023}
Lukas Renelt, Christian Engwer, and Mario Ohlberger.
\newblock An {O}ptimally {S}table {A}pproximation of {R}eactive {T}ransport
  {U}sing {D}iscrete {T}est and {I}nfinite {T}rial {S}paces.
\newblock In {\em Finite Volumes for Complex Applications X---Volume 2,
  Hyperbolic and Related Problems}, pages 289--298, Cham, 2023. Springer.

\bibitem{romor2023nonlinear}
Francesco Romor, Giovanni Stabile, and Gianluigi Rozza.
\newblock Non-linear manifold reduced-order models with convolutional
  autoencoders and reduced over-collocation method.
\newblock {\em J. Sci. Comp.}, 94(3):74, 2023.

\bibitem{romor2023friedrichs}
Francesco Romor, Davide Torlo, and Gianluigi Rozza.
\newblock Friedrichs' systems discretized with the {D}iscontinuous {G}alerkin
  method: domain decomposable model order reduction and {G}raph {N}eural
  {N}etworks approximating vanishing viscosity solutions.
\newblock {\em arXiv preprint arXiv:2308.03378}, 2023.

\bibitem{urban2023reduced}
Karsten Urban.
\newblock The {R}educed {B}asis {M}ethod in {S}pace and {T}ime: {C}hallenges,
  {L}imits and {P}erspectives.
\newblock In {\em Model Order Reduction and Applications}, pages 1--72.
  Springer Nature, 2023.

\end{thebibliography}

\bigskip

\appendix
\section{Bijectivity of the adjoint operator}\label{appendix:adjointTransportOperator}
In this section we prove that the conditions in Prop.\ref{prop:transport:poincare} imply the assumptions \ref{ass:injectivity} and \ref{ass:surjectivity}, respectively.

\begin{theorem}
  Let there be $\kappa>0$ with $c(x) \geq \kappa$ almost everywhere. Then, the adjoint operator $A_\circ^*$ defined in~\eqref{def:transport:adjointOperator} is injective on $\Ydense=C^\infty(\Omega)$ and the Poincaré-type inequality $\norm{v}_{\xcal} \lesssim \norm{A_\circ^*[v]}_{\xcal'}$ holds.
\end{theorem}
\begin{proof}
  We will prove the Poincaré-type inequality first. Notice, that this immediately implies injectivity of $A_\circ^*$.
  Using integration by parts we have:
  \begin{align*}
    A_\circ[v](v)
    &= (-\vec{b}\nabla v + cv,v)_{\Ltwo} + \norm{v}_{\LtraceOut}^2 \\
    &= (cv,v)_{\Ltwo} + \tfrac{1}{2}\norm{v}_{\LtraceIn}^2 + \tfrac{1}{2}\norm{v}_{\LtraceOut}^2\\
    &\geq \min\{\kappa,\tfrac{1}{2}\}\norm{v}_{\xcal}^2.
  \end{align*}
  On the other hand, we have
  \begin{equation*}
    |A_\circ^*[v](v)| \leq \norm{A_\circ^*[v]}_{\xcal'}\norm{v}_{\xcal}
  \end{equation*}
  by Cauchy-Schwarz and thus
  \begin{equation}
    \norm{v}_{\xcal}\leq \max\{\kappa^{-1}, 2\}\norm{A_\circ^*[v]}_{\xcal'}.
  \end{equation}
\end{proof}

\begin{theorem}
  Let $\vec{b}$ be $\Omega$-filling. Then, the adjoint operator $A_\circ^*$ defined in~\eqref{def:transport:adjointOperator} is injective on $\Ydense=C^\infty(\Omega)$ and the Poincaré-type inequality $\norm{v}_{\xcal} \lesssim \norm{A_\circ^*[v]}_{\xcal'}$ holds.
\end{theorem}
\begin{proof}
  By testing with $(0,v\restr{\Gout})$ we directly obtain
  \begin{equation}\label{eq:proofInjectivity:Ltrace}
    \norm{v}_{\LtraceOut}^2 = |A_\circ^*[v](0,v\restr{\Gout})| \leq \norm{A_\circ^*[v]}_{\xcal'}\norm{v}_{\LtraceOut}.
  \end{equation}

 Now, we multiply $v$ with the cutoff-function $\rho$ defined in~\cite[Lemma~A.2]{BrunkenSmetanaUrban} as
  \begin{equation*}
    \rho\in \Linfty, \qquad \rho(\xi(t,x)) := 2t.
  \end{equation*}
  This function then fulfills (a.e.)
  \begin{equation*}
    \vec{b}\nabla\rho \equiv 2 \quad\text{in}\;\Omega, \qquad \rho=0 \quad\text{on}\;\Gin \quad\text{and}\quad 0\leq 2T_{min} \leq \rho\leq 2T_{max} \quad\text{on}\;\Gout.
  \end{equation*}
  For any $v\in \Ydense$ we then have
  \begin{align*}
    A_\circ^*[v](\rho v, \rho v\restr{\Gout})
    &= (-\vec{b}\nabla v + cv, \rho v)_{\Ltwo} + (v,\rho v)_{\LtraceOut} \\
    &= \int_\Omega (\frac{1}{2} \vec{b}\nabla\rho + \rho c)v^2 \diff x + \int_{\Gout} (\rho - \frac{1}{2}\rho) v^2 |\vec{b}\vec{n}| \diff s \\
    &\geq \Lnorm{v}^2 + T_{min}\norm{v}_{\LtraceOut}^2
  \end{align*}
  On the other hand, we have
  \begin{align*}
    |A_\circ^*[v](\rho v, \rho v\restr{\Gout})|
    &\leq \norm{A_\circ^*[v]}_{\xcal'}\norm{\rho(v, v\restr{\Gout})}_{\xcal} \\
    &\leq \norm{A_\circ^*[v]}_{\xcal'}\norm{\rho}_{\Linfty}\norm{v}_{\xcal} \\
    &\leq 2\,T_{max}\norm{A_\circ^*[v]}_{\xcal'}\norm{v}_{\xcal}
  \end{align*}
  and thus by adding inequality~\eqref{eq:proofInjectivity:Ltrace} (if $T_{min} < 1$)
  \begin{align*}
    \Lnorm{v}^2 + \norm{v}_{\LtraceOut}^2
    &\leq 2\,T_{max}\norm{A_\circ^*[v]}_{\xcal'}\norm{v}_{\xcal} \\
    & \qquad + \max\{1-T_{min},0\}\norm{A_\circ^*[v]}_{\xcal'}\norm{v}_{\LtraceOut}\\
    &\leq (2\,T_{max}+\max\{1-T_{min},0\})\norm{A_\circ^*[v]}_{\xcal'}\norm{v}_{\xcal}
  \end{align*}
\end{proof}

\begin{theorem}
  Let one of the conditions in Prop.~\ref{prop:transport:poincare} hold. Then, the image of the adjoint operator $A_\circ^*: C^\infty(\Omega) \rightarrow \xcal'$ is dense in $\xcal'$.
\end{theorem}
\begin{proof}
  Let $w \in (A_\circ^*[\ycal])^\perp \subseteq \xcal'$ and $\mathcal{D} := C_0^\infty(\Omega)$. Denote by $(w_1,w_2) := R^{-1}_{\xcal_{L^2}}w$ the Riesz-representative of $w$ in $\xcal_{L^2}$. Then for all $v\in\mathcal{D}$
  \begin{align*}
    0 \;=\; (A_\circ^*[v],w)_{\xcal'}
    \; &=\; \Lscal{-\vec{b}\nabla v + cv}{w_1} + (v,w_2)_{\LtraceOut} \\
    \; &=\; (v, \vec{b}\nabla w_1 + cw_1)_{\mathcal{D}\times\mathcal{D}'}
  \end{align*}
  and thus $\vec{b}\nabla w_1 = -cw_1 \in \Ltwo$, \ie~$w_1\in H^1(\vec{b},\Omega)$. For arbitrary $v\in C^\infty(\Omega)$ we now obtain
  \begin{equation*}
    (v,w_1)_{\LtraceIn} + (v, w_1 - w_2)_{\LtraceOut} = 0
  \end{equation*}
  and therefore (since $v$ is arbitrary) $w_1\restr{\Gin}=0$, as well as $w_1\restr{\Gout} = w_2$. By the Poincaré-type inequality~\eqref{eq:transportPoincare} (using the flipped transport direction $\vec{\beta} := -\vec{b}$ and switched in-/outflow boundary parts in the adjoint operator) we obtain
  \begin{equation*}
    0 \;=\; \Lnorm{-\vec{\beta}\nabla w_1 + cw_1}^2 + \norm{w_1\restr{\Gin}}_{\LtraceIn}^2 \;=\; \norm{\tilde{A}^*[w_1]}_{\xcal'}^2 \;\gtrsim\; \norm{w_1}_{\xcal}^2
  \end{equation*}
  and thus $w_1 = 0$. With $w_2 = w_1\restr{\Gout} = 0$ we have finally shown $w=0$ and thus proved the claim.
\end{proof}

\section{A trace theorem for
  \texorpdfstring{$\Honeb$}{H1b}}\label{appendix:traceTheoremsH1b}
Unlike the classic $H^1(\Omega)$, the Sobolev-space $\Honeb$ does not admit a classic trace operator. The existence of traces depends on the structure of the velocity field $\vec{b}$ and its properties. The following theorem states that for $\Omega$-filling transport fields with a positive minimal traverse time a trace operator always exists:
\newcommand{\LtraceT}{L^2(\partial\Omega, T|\vec{b}\vec{n}|)}
\newcommand{\LtraceOutT}{L^2(\Gout, T|\vec{b}\vec{n}|)}
\begin{lemma}\label{appendix:lemma:traceHoneb}
  Let $\vec{b}$ be $\Omega$-filling and the minimal traverse time bounded away from zero. Then, there exists a linear and continuous trace operator
  \begin{equation*}
    \gamma_{\Honeb}: \Honeb \rightarrow \LtraceOut, \qquad \norm{\gamma_{\Honeb}}^2 \leq T_{min}^{-1}(2\,T_{max} + 1)
  \end{equation*}
  with $\gamma_{\Honeb}(v) = v\restr{\partial\Omega}$ for all $v\in C^1(\Omega)$.
\end{lemma}
\begin{proof}
  Consider the cutoff-function $\rho$ (similar to~\cite[Lemma~A.2]{BrunkenSmetanaUrban}) defined as
  \begin{equation*}
    \rho\in \Linfty, \qquad \rho(\xi(t,x)) := t.
  \end{equation*}
  This function then fulfills (a.e.)
  \begin{equation*}
    \vec{b}\nabla\rho \equiv 1 \quad\text{in}\;\Omega,
    \quad \rho=0 \quad\text{on}\;\Gin,
    \quad \rho(x) = T(x)\quad\text{on}\;\Gout,
    \quad 0\leq \rho\leq T_{max} \quad\text{in}\;\Omega.
  \end{equation*}
  Using integration by parts we obtain
  \begin{align*}
    2(\vec{b}\nabla v, \rho v)_{\Ltwo}
    &= -(\vec{b}\nabla\rho, v^2)_{\Ltwo} + \int_{\Gout} T(s) v^2 |\vec{b}\vec{n}|\diff s \\
    &\geq -\Lnorm{v}^2 + T_{min}\norm{v}_{\LtraceOut}^2
  \end{align*}
  and thus by Cauchy-Schwartz on the left hand side and using $\rho \leq T_{max}$ a.e. in $\Omega$
  \begin{equation}
    \norm{v}_{\LtraceOutT}^2 \leq T_{min}^{-1}(2\,T_{max} + 1)\norm{v}_{\Honeb}^2.\tag*{\hspace*{1mm}}
  \end{equation}
\end{proof}

\end{document}